\documentclass[hidelinks,onefignum,onetabnum]{siamart251216}

\usepackage{lipsum}
\usepackage{amsfonts}
\usepackage{graphicx}
\usepackage{epstopdf}
\usepackage{algorithmic}
\ifpdf
  \DeclareGraphicsExtensions{.eps,.pdf,.png,.jpg}
\else
  \DeclareGraphicsExtensions{.eps}
\fi

\newsiamremark{remark}{Remark}
\newsiamremark{hypothesis}{Hypothesis}
\crefname{hypothesis}{Hypothesis}{Hypotheses}
\newsiamthm{claim}{Claim}
\newsiamremark{fact}{Fact}
\crefname{fact}{Fact}{Facts}

\headers{Damped SWIFT}{Trevisani, L\'opez-Salas, Ben Hammouda and Oosterlee}

\title{A Damped SWIFT Method for European Option Pricing: Coefficients Decay, Truncation, and Error Analysis\thanks{Submitted to the editors DATE.
\funding{This work was funded 
by the grant ED431G 2023/01 of CITIC.}}}

\author{Davide Trevisani\thanks{Department of Mathematics and CITIC, University of A Coruña, Campus de Elviña s/n, A Coruña, 15071, Galicia, Spain
  (\email{davide.trevisani@udc.es}, \email{jose.lsalas@udc.es}).} \and Jos\'e Germ\'an L\'opez Salas\footnotemark[2] \and Chiheb Ben Hammouda\thanks{Mathematical Institute, Utrecht University, Utrecht, the Netherlands (\email{c.benhammouda@uu.nl}, \email{c.w.oosterlee@uu.nl}).}
\and  Cornelis W. Oosterlee\footnotemark[3]}

\usepackage{amsopn}

\ifpdf
\hypersetup{
  pdftitle={A Damped SWIFT Method for European Option Pricing: Coefficients Decay, Truncation, and Error Analysis},
  pdfauthor={D. Trevisani, J. G. L\'opez-Salas, C. Ben Hammouda and C. W. Oosterlee}
}
\fi

\usepackage{amssymb}
\usepackage{tikz}
\usetikzlibrary{arrows.meta,decorations.pathmorphing}
\usetikzlibrary{decorations.markings}

\usepackage{multicol}
\usepackage{multirow}
\usepackage{enumitem}
\usepackage{mathdots}

\newcommand{\R}{\mathbb{R}}
\newcommand{\Q}{\mathbb{Q}}
\newcommand{\Z}{\mathbb{Z}}
\newcommand{\E}{\mathbb{E}}
\newcommand{\ii}{\mathrm{i}}
\DeclareMathOperator*{\Res}{Res}
\newcommand{\C}{\mathbb{C}}
\newcommand{\dd}{\,\mathrm{d}}

\newtheorem{assumption}[theorem]{Assumption}

\allowdisplaybreaks

\begin{document}

\maketitle

% REQUIRED
\begin{abstract}
We introduce a damped variant of the Shannon Wavelet Inverse Fourier Technique (SWIFT) for pricing European options when the characteristic function of the underlying model is available. The key idea is to apply an exponential damping transformation to the payoff, which enables the direct computation of Fourier coefficients in the frequency domain without introducing an additional physical-domain truncation parameter. We provide a rigorous analysis of the decay of these coefficients by exploiting the singularity structure of the associated Fourier transforms. For light-tailed models, we obtain Gaussian-type decay estimates, while for semi-heavy and heavy-tailed models whose singularities are poles,  algebraic branch points, or logarithmic branch points, we derive exponential decay bounds with explicit polynomial prefactors. The resulting sharp bounds make it possible to truncate the Fourier series without relying on the cumulants of the underlying density, which are often unavailable or difficult to compute in practice. We further derive an error decomposition separating projection, truncation, and quadrature errors, and translate the analysis into practical rules for selecting the damping parameter, resolution level, and truncation range. Numerical experiments demonstrate that the proposed approach consistently improves the accuracy of the original SWIFT method while requiring a significantly smaller number of Fourier coefficients and remaining stable in cases where the undamped method deteriorates. 
\end{abstract}

% REQUIRED
\begin{keywords}
Fourier methods, option pricing, error analysis
\end{keywords}

% REQUIRED
\begin{MSCcodes}
68Q25, 60E10, 65T60, 91G20
% Analysis of algorithms and problem complexity
%   Characteristic functions; other transforms
%   Numerical methods for wavelets
%	Derivative securities (option pricing, hedging, etc.)
\end{MSCcodes}

\section{Introduction}

Fourier-based methods form a central class of techniques for the numerical valuation of financial derivatives, particularly when the characteristic function of the underlying stochastic process is available in closed form. Within this class, the Shannon Wavelet Inverse Fourier Technique (SWIFT), introduced in \cite{OrtizOosterlee2016}, provides a flexible framework based on wavelet approximations of the density. Related Fourier-based approaches include the Carr--Madan FFT method \cite{CarrMadan1999}, the Fourier-cosine (COS) method \cite{FangOosterlee2008}, and generalized Fourier-transform methods \cite{BayerBenHammoudaPapapantoleonSametTempone2022,bayer2024quasi,EberleinGlauPapapantoleon2010,lewis2001simple}.

The SWIFT method represents the density by Shannon scaling functions, i.e., dilations and translations of the sinc function, forming an orthonormal basis of $L^2(\R)$. This construction combines several features that make it attractive for option pricing applications. The Fourier transforms of the Shannon scaling functions have compact support, leading to structured coefficient representations and efficient numerical implementations. Moreover, the localized nature of the Shannon wavelet basis allows the approximation to be expressed in terms of translation coefficients whose magnitude reflects the contribution of different regions of the density.  This feature is particularly attractive when the density is asymmetric, heavy-tailed, or sharply peaked, as may occur for short maturities or under Lévy-type dynamics, where the choice of an appropriate computational domain and truncation strategy becomes more delicate. Unlike global expansions, where the approximation is determined by coefficients that depend on the entire domain, the wavelet representation provides a natural mechanism for identifying the regions that contribute most significantly to the approximation through the decay of the corresponding coefficients. The method has been successfully applied to a variety of option pricing problems,
%, including European, Bermudan, and discretely monitored barrier options; 
see \cite{OrtizOosterlee2016,MareeOrtizOosterlee2017}. 
%In particular, the locality of the basis has proven advantageous in recursive valuation procedures arising in early-exercise and path-dependent contracts, where repeated projections may amplify boundary effects. 
Exponential convergence with respect to (wrt) the wavelet scale has been established under suitable regularity assumptions.

Despite these attractive properties, several aspects of the SWIFT methodology remain insufficiently understood. In practical implementations, numerical parameters such as the wavelet scale, the number of translation terms, and the effective computational domain are often selected using heuristic considerations. While these choices have proven successful in many applications, a systematic framework for truncation, parameter selection, and error control is still largely absent. As discussed in \cite{LeFloch2024}, the performance of the method can be sensitive to these choices, particularly for densities exhibiting strong asymmetry, heavy tails, or pronounced concentration effects. Similar challenges have recently been investigated in detail for the COS method; see \cite{Junike2024,JunikePankrashkin2022,JunikeStier2025}.  These works demonstrate that parameter selection can play a decisive role in the overall accuracy and efficiency of spectral approximation methods and motivate a corresponding investigation for the SWIFT framework. In particular, our goal is to understand how truncation and parameter selection can be related to the decay properties of the underlying wavelet coefficients.

To address this question, we introduce a damped formulation of SWIFT and analyze the resulting density and payoff coefficients directly in Fourier space.  The damping transformation facilitates a detailed study of their asymptotic decay, allowing truncation and parameter selection to be related to the singularity structure of the characteristic function.  Moreover, the damping transformation yields a fully Fourier-space representation of the density and payoff coefficients.  In contrast to recent COS analyses, which derive truncation rules from real-variable estimates on tail decay and regularity, our approach is based on the complex singularity structure of the characteristic function. Poles and algebraic/logarithmic branch points nearest the real axis determine the asymptotic decay of the SWIFT coefficients: the distance of the dominant singularity from the real axis governs the exponential decay rate, while its local structure determines the polynomial prefactor.

Building on this singularity analysis, we derive explicit decay estimates for the product of the SWIFT coefficients, %and provide a parameter-selection strategy inferred directly from their decay wrt the translation parameter. 
and propose coefficient-based truncation rules, in contrast to existing numerical implementations that rely on physical-space truncations, such as cumulant-based approaches.
%Existing implementations typically rely on cumulant-based interval selection. 
%Our approach yields coefficient-based truncation rules that do not rely on cumulant heuristics. 
%Within the proposed framework, the truncation range is determined by the decay of the wavelet coefficients. 
This also distinguishes the present framework from recent damped COS approaches, where truncation is formulated through global cosine-series approximations on a prescribed interval, without exploiting the decay of the payoff coefficients.
Our error analysis 
%provides additional insight into the approximation error and 
offers practical guidance for choosing numerical parameters in a mathematically justified manner. 
%This also distinguishes the present framework from recent damped COS approaches, where truncation is formulated through global cosine-series approximations on a prescribed interval. %In contrast, the wavelet framework derives truncation directly from the decay of the damped Shannon-wavelet coefficients.

%Recent developments in higher-dimensional Fourier pricing have highlighted the importance of the regularity structure and rigorous control of truncation and approximation errors \cite{BayerBenHammoudaPapapantoleonSametTempone2022,bayer2024quasi}. In such settings, computational efficiency is often highly sensitive to the choice of truncation and approximation parameters. Moreover, formulations that allow the relevant quantities to be computed directly in Fourier space are particularly attractive, as they expose the analytic structure of the integrands and can lead to more effective quadrature-based methods. 
Before extending the damped SWIFT to multidimensional problems, it is therefore important to understand the mechanisms governing coefficient decay, truncation, and approximation accuracy already in one dimension. The framework developed here is intended to provide such a foundation and to support the development of parameter-selection strategies for more general wavelet-based Fourier methods.

The main contributions of the paper are as follows. First, we introduce a damped SWIFT formulation for European option pricing in models where the characteristic function of the log-price process is available. The proposed formulation enables the density and payoff coefficients to be computed directly in Fourier space without introducing an additional physical-domain truncation parameter. Second, we derive coefficient-decay estimates for the damped density and payoff coefficients. 
%For light-tailed models we obtain Gaussian-type decay estimates, while for models whose characteristic functions possess algebraic singularities we derive exponential decay bounds with explicit polynomial prefactors. 
Third, we use these decay estimates to derive explicit truncation bounds and coefficient-based truncation rules that do not rely on cumulants of the underlying density. Fourth, we derive an error decomposition separating projection, truncation, and quadrature errors. Finally, we translate the analysis into practical guidelines for selecting the damping parameter, the resolution level, and the truncation range.  The numerical experiments validate the theoretical findings and illustrate how the proposed coefficient-decay analysis leads to effective truncation strategies and improved computational efficiency.

The remainder of the paper is organized as follows. In \cref{sec:DSWIFT} the damped SWIFT framework is introduced. \Cref{sec:Asymptotic decay of the payoff and density coefficients with respect to the translation parameters} derives bounds for the Fourier coefficients, while \cref{Error Decomposition} presents the error decomposition and the associated parameter-selection framework. Finally, \cref{sec:numREsults} contains numerical experiments. % that validate the theoretical results and illustrate the practical performance of the proposed methodology.

\section{Damped SWIFT}
\label{sec:DSWIFT}

\paragraph{Notation}
\begin{enumerate}[label=(\Roman*)]
    \item $\ii$ is the imaginary unit number. For $z\in\C$, $\Re(z)$ and $\Im(z)$ denote its real and imaginary parts, respectively.
    \item The extended Fourier Transform (FT) of a function $h$, denoted by $\hat{h}$, is defined by $\hat{h}(z) = \int_{\R} h(y) e^{-\ii z y} \dd y$, $z\in\C$, whenever the integral is well defined.
    \item $B(z_0,\epsilon)$ denotes the open ball $\{ z\in \C :\, |z-z_0|<\epsilon\}.$
    \item For $1\le p\le\infty$, $L^p(\R)$ denotes the usual Lebesgue space, and $\|\cdot\|_p$ the usual norm. 
    \item We write $\mathbf 1_A$ for the indicator function of a set $A$.
    \item  The sign function $\operatorname{sgn}(u) $ is equal to $1$ for $u\geq 0$ and $-1$ for $u<0$.
    %\item In proofs, we denote as $C$ any positive constant.
    \item 
    We say that a property holds piecewise for a function $h:\mathbb{R}\to\mathbb{R}$ if there exists a finite partition of $\mathbb{R}$ into intervals such that the property is satisfied on each interval. Throughout this work, we consider functions that are, in this sense, piecewise monotone, piecewise continuous, piecewise $C^1$, or piecewise $C^2$.
\end{enumerate}

Let $(\Omega,\mathcal F,\mathbb F,\Q)$ 	be a filtered probability space, where $\Q$ is a risk-neutral pricing measure and $\mathbb F=(\mathcal F_t)_{0\le t\le T}$ is a filtration satisfying the usual conditions.
We compute the price of European options with payoff $P$ when the log-asset price follows a stochastic model $X_t:=\log S_t$, whose characteristic function is known at time-to-maturity $T>0$. 
%For notational simplicity, we set the valuation time equal to zero, so that $T$ denotes the time to maturity. 
If the payoff is integrable under $\mathbb{Q}$, the value of a European option is then given by
\begin{equation}\label{eq:valuation}
		v(x)
		=
		e^{-r_0T}\E^\Q\!\left[P(X_T)\mid X_0=x\right]
		=
		e^{-r_0T}\int_\R P(y)f(y\mid x) \dd y,
	\end{equation}
with $x$ denoting the current log-asset price, $f(\cdot\mid x)$ the conditional density of $X_T$ given $X_0=x$, and $r_0$ being the risk-free rate. For instance, a European call corresponds to $P(y)=(e^y-K)^+$, whereas a cash-or-nothing digital option with barrier \(B>0\) corresponds to $P(y)=\mathbf 1_{\{y>\log B\}}$.
\begin{assumption}
    \label{ass: f and P regularity. P jumps}
    In this work, we require that both the density $f$ and the payoff $P$ are piecewise of class $C^1$ and have a finite number of critical points. The payoff $P$ is defined in the entire real line and is allowed to jump, provided that the size of its jumps is uniformly bounded.
\end{assumption}
In order to exploit the improved structure available in Fourier space, and to avoid introducing additional truncation parameters in the physical domain as in
%Instead of computing physical-space integrals for density and payoff coefficients, we exploit the improved regularity in Fourier space and map computations there, thereby avoiding the need to introduce truncation parameters for the physical domain, as done in 
\cite{OrtizOosterlee2016,colldeforns2017two}, we introduce exponential damping. 
%To ensure the well-posedness of the valuation formula and the subsequent Fourier-space coefficient construction, we introduce exponential damping.
%Our approach benefits from the compact support of wavelet Fourier transforms, eliminating the need for inverse transformations as required in \cite{BayerBenHammoudaTempone2023}.
For $\alpha \in \R$, we define the damped payoff $P_{\alpha}(y) := e^{-\alpha  y} P(y)$, and density $f_{\alpha}(y \mid x) := e^{\alpha y} f(y \mid x)$. Then,  $P(y) f(y \mid x) = P_{\alpha}(y) f_{\alpha}(y \mid x)$, and 
\begin{equation}\label{eq:valuation-damped}
	v(x)=e^{-r_0 T}\int_{\R}P_{\alpha}(y)\, f_{\alpha}(y \mid x)\, \dd y.
\end{equation}
We now consider the admissible strips
$\delta_P := \{\alpha \in \R\;|\;P_{\alpha} \in L^1(\R)\}$ and $ \delta_f :=\{\alpha \in \R\;|\;	f_{\alpha} \in L^1(\R)\}$ associated with the payoff and the density, respectively. Notice that $\alpha \in \delta_f$ if and only if $
		\E^\Q\!\left[S_T^\alpha\mid S_0=e^x\right]
		<\infty .
		$
\begin{assumption}
\label{ass:noempty strip intersection}
 We assume that  $ \alpha\in \delta_V := \delta_P \cap \delta_f \neq \varnothing$, and  $f_\alpha, P_\alpha \in L^2(\R)$.
\end{assumption} 
Observe that $\delta_V\neq \varnothing$ combined with \cref{ass: f and P regularity. P jumps} guarantees the well-posedness of the valuation formula \eqref{eq:valuation-damped}, as they imply that $f_\alpha\in L^1(\R)$, and
$P_\alpha \in L^1(\R)\cap L^\infty(\R)$. 
We also highlight that, in this work, the $L^2$ condition in \cref{ass:noempty strip intersection} is required only for simplicity and can be omitted. Indeed, the $L^2$ setting implies the Parseval identity, although for Shannon wavelets this identity also holds in a $L^1$ setting plus \cref{ass: f and P regularity. P jumps}, as proved in \ref{sm:proof Parseval shannon}. In fact, $f_\alpha \in L^2(\R)$ is not satisfied by densities having blow-ups that are not square-integrable, as in some configurations of the Variance Gamma distribution. 

\subsection{SWIFT approximation}
\label{subsec:swift-approximation}
Let $m\in \Z$ be a fixed resolution level, and consider the Shannon scaling functions 
	\begin{equation*}%\label{eq:phi-def}
		\phi_{m,k}(y)
		:=
		2^{m/2}\operatorname{sinc}(2^m y-k),
		\qquad
		\operatorname{sinc}(x):=\frac{\sin(\pi x)}{\pi x},
		\qquad
		k\in\mathbb Z.
	\end{equation*}
    As all the $\phi_{m,k}\in L^2(\R)$, their Fourier transforms can be defined as in \cite[Theorem 9.13]{rudin1987real}, and it holds that
    $
		\hat{\phi}_{m,k}(u)
		=
		2^{-m/2}
		e^{-\ii k u/2^m}
		\mathbf 1_{[-2^m\pi,2^m\pi]}(u).
	$
    An introduction to the Shannon wavelets can be found in \cite{OrtizOosterlee2016} and references therein.
In order to compute the option price we now have the following steps.
\paragraph{Step 1: Density coefficients approximation}
The damped density $f_\alpha$ is approximated by a combination of Shannon scaling functions given by
	\begin{equation}\label{eq:density-proj}
		f_{1,\alpha}(y\mid x):=
		\sum_{k\in\Z} D^\alpha_{m,k}(x)\,\phi_{m,k}(y),\quad
		D^\alpha_{m,k}(x):=\int_{\R} f_\alpha(y\mid x)\,\phi_{m,k}(y)\dd y .
	\end{equation}
	Analogously, the payoff coefficients are defined as 
	\begin{equation*}%\label{eq:V-def}
		V^\alpha_{m,k}
		:=
		\int_\R P_\alpha(y)\phi_{m,k}(y) \dd y.
	\end{equation*}
    To compute these coefficients, we do not approximate the cardinal sine with a finite combination of cosines via Vieta’s formula, as in \cite{OrtizOosterlee2016}, and we do not use a midpoint quadrature rule combined with exponential approximation, as in \cite{colldeforns2017two}. Instead, we map the problem to the frequency space using the Parseval identity. Namely, as $f_\alpha \in L^2(\R)$, we have
    %(see \cite[Theorem 9.13 part (b)]{rudin1987real})
	\begin{equation}\label{eq:D-fourier-real}
		D^\alpha_{m,k}(x)=
		\frac{1}{2\pi}
		\int_\R
		\hat{f}_\alpha(u\mid x)
		\overline{\hat{\phi}_{m,k}(u)}
		\dd u=
		\frac{2^{-m/2}}{2\pi}
		\int_{-2^m\pi}^{2^m\pi}
		\hat{f}(u+\ii\alpha\mid x)
		e^{\ii k u/2^m}		
		\dd u,
	\end{equation}
	and analogously,
	\begin{equation}
		V^\alpha_{m,k}=
		\frac{1}{2\pi}
		\int_\R
		\hat{P}_\alpha(u)
		\overline{\hat{\phi}_{m,k}(u)}
		\dd u
		=
		\frac{2^{-m/2}}{2\pi}
		\int_{-2^m\pi}^{2^m\pi}
		\hat{P}(u-\ii\alpha)
		e^{\ii k u/2^m}
		\dd u.
		\label{eq:V-fourier}
	\end{equation}
We denote by $D^{\alpha,Q}_{m,k}(x)$ and $V^{\alpha,Q}_{m,k}$ the quadrature approximations of \eqref{eq:D-fourier-real} and \eqref{eq:V-fourier}, respectively.
% To enable numerical integration via quasi-Monte Carlo (QMC) methods, we map the integration problem in \eqref{eq:D_H} to $[0,1]$. If we perform the change of variables $t = \frac{w}{2^{m+1}\pi} +\frac12$, and we take into account that the density coefficients are real-valued, then
% \begin{align}
% D^{\alpha}_{m,k}(x)&=2^{\frac{m}{2}} \int_{0}^{1} \Re \left\{ \hat f_{\alpha}\!\big(2^{m}\pi(2t-1)|x\big)\,
% e^{\,\ii\pi k (2 t-1)}\right\}dt. \label{eq:D_H_unit}
% \end{align}
% Taking into account that
% \begin{align*}
%  \hat{f}_{\alpha}(w|x) =& \int_{\R} f_{\alpha}(y|x) e^{-\ii w y} dy = \int_{\R} e^{\alpha y }f(y|x) e^{-\ii w y} dy \\
%  =&\int_{\R} f(y|x) e^{-\ii (w + \alpha \ii) y} dy = \hat{f}(w+\alpha \ii|x),
% \end{align*}
% one readily obtains
% \begin{align}
% D^{\alpha}_{m,k}(x)&=2^{\frac{m}{2}} \int_{0}^{1} \Re \left\{ \hat{f}\big(2^{m}\pi(2t-1)+\alpha \ii|x\big)\,
% e^{\,\ii\pi k  (2t-1)}\right\}dt.\label{eq:D_H_unit_2}
% \end{align}
% We denote by $D^{\alpha, \text{Q}}_{m,k}(x)$ the quadrature-approximation of \eqref{eq:D_H_unit_2}. Let $g_{m,\alpha}(t|x) :=  \hat{f}\!\big(2^{m}\pi (2 t-1)+\alpha \ii|x\big) $.

Compared to \cite{OrtizOosterlee2016,colldeforns2017two}, our Fourier-space coefficient computation avoids additional physical-domain truncation parameters, whose selection can be delicate for high-dimensional models, complex payoff structures, or cumulant-based interval choices.
%Compared to \cite{OrtizOosterlee2016,colldeforns2017two}, another advantage of our approach when computing the density and the payoff coefficients 
%%(as in \eqref{eq:D_H}-\eqref{eq:V_payoff}) 
%is that we do not need to introduce additional truncation parameters in the physical domain, which are often difficult to determine as the dimension increases or for complex models and payoff structures (since it depends on cumulants and other parameters). %This is thanks to working in the Fourier space, benefiting from the compact support of the Fourier transforms of the Shannon father wavelets.

\paragraph{Step 2: Truncation of the series}
Let $l,u\in\Z$ with $l< u$, and define $\Lambda:=[l,u]\cap\Z$. 
Truncating \eqref{eq:density-proj} we get
$$ f_{2,\alpha}(y\,|\,x) :=\sum_{k=l}^u D^{\alpha}_{m,k}(x)\,\phi_{m,k}(y).$$
Note that $\Lambda$ need not be symmetric around zero and may be shifted according to the location of dominant density-payoff coefficient products.
The choice of this truncation range is discussed in \cref{sec:algorithm}.

\paragraph{Step 3: Option-price approximation}
Substituting the successive density approximations into \eqref{eq:valuation-damped} yields three approximations, separating projection, truncation, and quadrature errors:
    \begin{align}
		%v(x)
		%&=e^{-r_0T}\int_{\R} P_\alpha(y)\,f_\alpha(y\mid x)\dd y, \notag \\
		v_m(x)
		&:=e^{-r_0T}\int_{\R} P_\alpha(y)\,f_{1,\alpha}(y\mid x)\dd y
		=e^{-r_0T}\sum_{k\in\Z} D^\alpha_{m,k}(x)\,V^\alpha_{m,k}, \label{eq:v_m(x)}\\
		v_{m,\Lambda}(x)
		&:=e^{-r_0T}\int_{\R} P_\alpha(y)\,f_{2,\alpha}(y\mid x)\dd y
		=e^{-r_0T}\sum_{k\in \Lambda} D^\alpha_{m,k}(x)\,V^\alpha_{m,k},\label{eq:v_m,Lambda(x)} \\
		v_{m,\Lambda,Q}(x)
		&:=e^{-r_0T}\sum_{k=l}^{u} D^{\alpha,Q}_{m,k}(x)\,V^{\alpha,Q}_{m,k}. \label{eq:v_m,Lambda,Q(x)}
	\end{align}
	%Here \eqref{eq:vm} uses the orthonormal-expansion identity %$\int_{\R}P_\alpha\,\phi_{m,k}\,dy=V^\alpha_{m,k}$; \eqref{eq:vmL} replaces the exact series by its truncation \eqref{eq:truncated-density}; and \eqref{eq:vmLQ} replaces both exact coefficient families by their quadrature counterparts. 
    The total error
	$v-v_{m,\Lambda,Q}$ decomposes accordingly into the projection error due to $v-v_m$, the
	truncation error arising from $v_m-v_{m,\Lambda}$, and the quadrature error
	$v_{m,\Lambda}-v_{m,\Lambda,Q}$, analyzed in \cref{Error Decomposition}.

%The resulting approximation of $v(t_0, x)$ in \eqref{eq:QoI}   is then given by,
%	\begin{align*}
%		v_1(t_0, x)&:= e^{-r(T-t_0)}\int_\R P_\alpha(y) f_{1,\alpha}(y\,|\,x) \,dy \\
%        &\approx  e^{-r (T-t_0)} \int_{\R} P_{\alpha}(y) f_{2, \alpha}(y|x)  dy =  e^{-r (T-t_0)}	\sum_{k=l}^u D^{\alpha,\text{Q}}_{m, k}(x) \int_{\R} P_{\alpha}(y) \phi_{m, k}(y)  dy \\
%		&=  e^{-r (T-t_0)}	\sum_{k=l}^u D^{\alpha}_{m,k}(x)\, V^{\alpha}_{m,k}\approx  e^{-r (T-t_0)}	\sum_{k=l}^u D^{\alpha,\text{Q}}_{m,k} (x)\, V^{\alpha,\text{Q}}_{m,k}:=	v_2(t_0, x).
%	\end{align*}

\section{Coefficients decay with respect to the translation parameter}\label{sec:Asymptotic decay of the payoff and density coefficients with respect to the translation parameters}
    
In this section, we study the decay of the SWIFT coefficients wrt the translation parameter. First, consider the following: for $h\in L^1(\R)$, $r>0$ and $y\in \R$ let
	\begin{equation}\label{eq:IR-def-rewrite}
		I_r[\hat{h}](y):=\frac{1}{2\pi}\int_{-r}^{r}\hat{h}(u) e^{\ii u y}\dd u,\qquad y\in\R.
	\end{equation}
Observe that if $r=2^{m}\pi$, $y_k=k/2^{m}$ then $D^\alpha_{m,k}= 2^{-m/2} I_r[\hat{f}_\alpha](y_k)$ and $V^\alpha_{m,k}= 2^{-m/2} I_r[\hat{P}_\alpha](y_k)$. Therefore, up to a factor $2^{-m/2}$, the decay of the SWIFT coefficients wrt $k$ will be implied by the one of \eqref{eq:IR-def-rewrite} wrt $y\in \R$. On the other hand, $I_r[\hat{h}](y)$ can be seen as a Fourier inversion formula for $h$, truncated in the interval $[-r,r]$. Therefore, any bound for $|I_r[\hat{h}](y)|$  will reasonably depend on $h(y)$ plus an error due to truncation in the frequency space. 
%In this section, we provide sufficient assumptions for which this intuition is mathematically true. 

However, when $h=f_\alpha$, $h$ might not be available in a closed formula, and the tail behaviour of the undamped density function has to be inferred from the analytical continuation of its FT. If $\hat{f}_\alpha$ can be extended to an entire function, then the decay of $f_\alpha(y)$ must be faster than $e^{-C|y|}$ for all $C>0$, which is the case of \cref{subsec:ligh-tailed density}, where we derive Gaussian-type decay. On the other hand, if $\hat{f}_\alpha$ is a meromorphic function or it has branch-points, then exponential-polynomial decay can be proved, as we show in \cref{subsec:semi-heavy tailed distr}. %Namely, the exponential decay rate is determined by the distance of the nearest singularity to the real axis, while the algebraic prefactor is determined by the order of the singularity.

The error due to truncation in the frequency space can be studied using integration by parts, and as observed in \cite[Section 5.2]{miller2006applied} leads to an algebraically small contribution. The following result shows this fact under quite minimal assumptions.
\begin{lemma}\label{lem:fourier box value}
		Let $h\in L^1(\R)$ be real-valued, and assume that there exists $r>0$ such that $\hat{h}\in C^{N+1}(|u|>r)$ for some $N\in \mathbb{N}$, $\lim_{|u|\to \infty} \hat{h}^{(j)}(u)=0$ for $0\leq j\leq N$, and that $\hat{h}^{(N+1)} \in L^1(|u|>r)$.  Then, for $y\ne0$,
		\begin{equation}\label{eq:fourier box value}
				\frac{1}{2\pi}\int_{|u|\ge r}\hat{h}(u) e^{\ii u y}\dd u
				=
				-\frac{1}{\pi y}
				\sum_{n=0}^{N}
				\Im\!\left[\ii^n e^{\ii r y}\hat{h}^{(n)}(r) \right] y^{-n}
				+\mathcal R_N(r,y),
		\end{equation}
		with
		\begin{equation}\label{eq:four box remainder-bound}
				|\mathcal R_N(r,y)|
				\le
				\frac{1}{2\pi|y|^{N+1}}
				\int_{|u|\ge r}|\hat{h}^{(N+1)}(u)|\dd u .
		\end{equation}
	\end{lemma}
\begin{proof}
  See \cref{app:lem:fourier box value}.
\end{proof}
In this work, \cref{lem:fourier box value} is applied to $f_\alpha$ and $P_\alpha$ with $N=0$ and $r=2^m \pi$. Hence, for $k\in \Z\setminus \{0\}$, and $y_k := k/2^m$ we obtain
    \begin{align*}
    &\hspace{-0.2cm}\left|\frac{1}{2\pi}\int_{|u|\ge 2^m \pi} \hat{f}_\alpha (u) e^{\ii u y_k}\dd u\right|  \leq 
        %\frac{(-1)^{k+1}}{\pi y}
        \varepsilon_D(m)|y_k|^{-1} + O(y_k^{-2}),\,\varepsilon_D(m)=\frac{1}{\pi} \left|\Im\left[ \hat{f} (2^m \pi+ \ii \alpha)\right]\right|,\\
%&=\varepsilon_D(m)\left|\frac{k}{2^m}\right|^{-1} + O(k^{-2}), \\
        &\hspace{-0.2cm} \left|\frac{1}{2\pi}\int_{|u|\ge 2^m \pi} \hat{P}_\alpha (u) e^{\ii u y_k}\dd u\right|  \leq 
        %\frac{(-1)^{k+1}}{\pi y}
        \varepsilon_V(m)|y_k|^{-1} + O(y_k^{-2}),\varepsilon_V(m) = {\frac{1}{\pi} \left|\Im\left[ \hat{P} (2^m \pi- \ii \alpha)\right]\right|}.%\\
        %&=\varepsilon_V(m)\left|\frac{k}{2^m}\right|^{-1} + O(k^{-2}).
    \end{align*}
However, it should be noted that sufficient conditions for \cref{lem:fourier box value} can be framed in terms of admissible strips. In fact, we have the following result.

\begin{lemma}
    \label{lem:four box assumption with strip payoff}
    Let $h$ be a real-valued piecewise $C^2$ function. Consider $\alpha$ in the interior of $\delta_{h}\cap \delta_{h^{(1)}}\cap \delta_{h^{(2)}}$, with $h^{(1)}$ and $h^{(2)}$ being the first and second derivatives of $h$, respectively. Then $h_\alpha$ satisfies the assumptions of \cref{lem:fourier box value}, with $N=0$ and $r>0$.
\end{lemma}
\begin{proof} 
See \cref{app:lem:fourier box value}.
\end{proof}
Since payoff functions are explicitly known in closed form and, in most applications, are piecewise smooth, \cref{lem:four box assumption with strip payoff} provides a convenient way to verify that the assumptions of \cref{lem:fourier box value} are satisfied by the damped payoff.
On the other hand, for density functions where only the FT is known, we provide the following sufficient conditions, also well-known in the literature (see for example \cite[IX.3]{simon1975methods}).
\begin{lemma}
\label{lem:four box assumption with strip density}
Let $h\in L^1(\R)$ be a real-valued function and $\alpha$ in the interior of $\delta_h$. If there exists $\epsilon>0$ such that 
$\sup_{|w-\alpha|< \epsilon}\|\hat{h}(\cdot + \ii w)\|_1 <\infty$, then $h_\alpha$ satisfies the assumptions of \cref{lem:fourier box value} for any $N\in \mathbb{N}$ and $r>0$.
\end{lemma}
\begin{proof}
See \cref{app:lem:fourier box value}.
\end{proof}

\subsection{Light-tailed densities}
\label{subsec:ligh-tailed density}
We now focus on FT with Gaussian decay, such as those obtained from Geometric Brownian Motion (GBM). 
\begin{proposition}
\label{prop:GBM tail bound}
Let $\hat{h}$ be an entire function and assume that, for some $C,\gamma>0$,
		\begin{equation}\label{eq:gaussian-growth-assumption}
			|\hat{h}(u)|\le C e^{-\gamma(\Re^2( u)-\Im^2(u))},\qquad u\in\C.
		\end{equation}
		Then, for every $r>0$ and $y\ne0$,
		\begin{equation}\label{eq:est coeff light tailed}
				|I_r[\hat{h}](y)|
				\le
				\frac{C}{2\sqrt{\pi\gamma}}e^{-\frac{y^2}{4\gamma}}
				+\frac{1}{\pi}\, \left| \Im\left[e^{\ii r  y}\hat{h}(r)\right]\right| |y|^{-1} + O(y^{-2}).
		\end{equation}
\end{proposition}
	\begin{proof} See \ref{sm:proofsLightTail}. 
	\end{proof}
\begin{corollary}
\label{cor:gbm-decay-}
Assume that the spot price $S_t$ is described by a GBM 
$\dd S_t = (r_0-q)S_t \dd t + \sigma S_t \dd W_t$.
Then the FT of $\log S_T$ is given by  $\hat{f}(u) = \exp(-\ii u \mu - \gamma  u^2 )$ where $\mu = \log\left(\frac{S_0}{K}\right)+(r_0-q-\frac12\sigma^2)T$, and $\gamma=\frac12\sigma^2 T>0$. 
For all $y_k:=k/2^m$, $y_{k,\mu}:=y_k-(\mu+2 \gamma \alpha)\neq 0$, up to $O(y_{k,\mu}^{-2})$,  we have
		\begin{equation}\label{eq:gbm-coeff-bound}
				|D^\alpha_{m,k}|
				\le
				2^{-m/2}
				\left(
				e^{\alpha\mu+ \gamma\alpha^2}\frac{e^{-\frac{y_{k,\mu}^2}{4\gamma}}}{2\sqrt{\pi \gamma}}
				+\frac{1}{\pi} \left|\Im\left[ \hat{f}(2^m \pi + \ii \alpha)\right]\right||y_{k,\mu}|^{-1}
				\right).
		\end{equation}
\end{corollary}
\begin{proof} See \ref{sm:proofsLightTail}. 
\end{proof}

%See SM2 for the application of these results to GBM dynamics. 
Notice that the bound \eqref{eq:gbm-coeff-bound} exhibits two different types of decay: the first term behaves as the density function%of a normal distribution
, while the second one goes as $|y_{k,\mu}|^{-1}$ times a constant that depends on $\hat{f}_\alpha(2^m\pi )$. In the following section, we prove that this type of bound also holds for fat-tailed distributions.

\subsection{Payoff functions and semi-heavy/heavy tailed densities}
\label{subsec:semi-heavy tailed distr}
We now consider functions whose Fourier transforms have isolated singularities away from the real axis. This setting covers damped payoffs and several (semi)-heavy-tailed models.
Exponential decay of Fourier coefficients is linked to analytic continuation and the singularities of the FT. The analysis shows that each isolated singularity contributes an exponentially small term. However, only the singularities closest to the real axis significantly affect the asymptotic behavior, while others are negligible. The exponential decay rate is governed by the imaginary part of the nearest singularity, and the polynomial prefactor depends on the order of the singularity.
%Exponential decay of the Fourier coefficients is linked to analytic continuation of $\hat f$ into a strip around the real axis. We assume that every singularity of the FT function is isolated and is either a pole or a branch point.  We start by analyzing the case of poles and then move on to branch points. In both cases, each isolated singularity produces an exponentially small contribution. When comparing the contributions from all singularities, only those closest to the real axis are significant; the remaining singularities are exponentially smaller in comparison. Consequently, the asymptotic behavior of the Fourier coefficients is fully determined by the singularity or singularities closest to the origin in the half-plane corresponding to the sign of $x$: those in the lower half-plane govern the behavior as $x \to -\infty$, while those in the upper half-plane determine the behavior as $x \to +\infty$. The imaginary part of the closest singularity will give the exponential decay rate, while the order of the pole/branch point will determine the polynomial prefactor.
\begin{definition}
For $ u_n \in  \C $ and $\theta\in (0,\pi)$, we denote by $ \gamma(u_n, \theta) $ the half-line $\{ u_n+  \operatorname{sgn}(\Im(u_n))e^{\ii \theta} t : t > 0  \}$. 
In case $ \theta = \frac{\pi}{2} $, the half-line is denoted as $ \gamma_{u_n} $, and it is called a branch cut starting at the branch point $u_n$.
\end{definition}
A branch cut $\gamma_{u_n}$ is parallel to the imaginary axis and the set $\C \setminus \gamma_{u_n}$ is simply connected. Fixed a definition of logarithm on this set, then for any $ \rho\in \C $, the function $u\mapsto(u-u_n)^{-\rho}$ is well-defined. 
\begin{definition}
%[Following \cite{miller2006applied}]
Let $ \hat{h}$ be defined on $ \C \setminus \{  \gamma_{u_1}, \allowbreak \ldots,\allowbreak \gamma_{u_N} \} $ where $\gamma_{u_1}, \ldots , \gamma_{u_N}$ are branch cuts. We say that $\hat{h}$ has algebraic branch points $u_1, \ldots , u_N$ if it has the form
$$ \hat{h}(u)= g(u) \prod_{n=1}^{N} (u-u_n)^{-{\rho_n}},$$
where $ g:\C\to \C$ is an entire function. The complex numbers $\rho_1, \ldots , \rho_N \in  \C \setminus \{ 0 \}$ are called orders of $ \hat{h} $. If close to some branch point $u_n$, we have that 
$\hat{h}(u)=[(u-u_n)^{-\rho}  \log(u-u_n)]g(u),$
with $-\rho\in \mathbb{N}$, then we say that the branch point is logarithmic of order $\rho$.
\end{definition}
If $ u_n $ is an algebraic branch point for $ \hat{h}$, we can find an open ball $ B(u_n,\epsilon) $ such that $ g(u):=(u-u_n)^{\rho}\hat{h}(u)$  is analytic in $ B(u_n,\epsilon) $.
In particular, if $\rho \in  \mathbb{N}$, then $ u_n $ is a pole of order $ \rho $ for $ \hat{h} $.
We now state the estimation theorem.
\begin{theorem}
  \label{teo:estimate fat tails}
Let $h\in L^1(\R)$ be a piecewise $C^1$ real-valued function, whose FT can be extended on $ \C \setminus \{ \gamma_{u_1}, \ldots , \gamma_{u_N} \}$ as  $ \hat{h} = Z + H $ where $H$ is an entire function and $Z$ has $ N $ algebraic or logarithmic branch points $ u_1, \ldots , u_N $ of orders $ \rho_1, \ldots , \rho_N \in  \C\setminus \{ 0 \}$.
Suppose also that
\begin{itemize}
  \item[i)]  $\hat{h}$ has at most polynomial growth, say $|\hat{h}(u)|\leq C (1+ |u|^q)$ for some $C,q>0$.
 \item[ii)] We can find $ \theta \in  (0,\frac{\pi}{2})$  such that 
 $$\sup_{u \in  \gamma(R, \theta)}\left(|\hat{h}( u)|+ |\hat{h}( -u)|\right), \quad   \sup_{u \in  \gamma(-R, \pi-\theta)}\left(|\hat{h}( u)|+ |\hat{h}( -u)|\right)$$
 tend to zero as $ R\to + \infty $.
  \item[iii)] There are $ a,b >0 $ and $ \gamma^+, \gamma^- >0$ such that
\begin{equation*}
  %\label{eq:assumption singularities position}
  %\begin{cases}
  \left\{
  \begin{aligned}
    &\Im(u_n)>0 \implies \ \Im(u_n)\geq b, \text{ and } \Im(u_n)=b \implies  \Re(\rho_n)\leq  \gamma^+, \\
    &\Im(u_n)< 0 \implies \ \Im(u_n)\leq  -a,\text{ and }  \Im(u_n)=-a \implies \Re(\rho_n)\leq  \gamma^- .\\
  %\end{cases}
  \end{aligned}
  \right.
\end{equation*}
\end{itemize}
%Then there exist $C>0 $ and $ R>0 $ such that
%\begin{equation}
%    \label{eq:estimate fat tails}
%    \begin{aligned}
%      |f(x)|&\leq  C x^{\gamma^+ -1}e^{-bx} ,\qquad x>  R, \\ |f(x)|&\leq  C |x|^{\gamma^- -1}e^{-a|x|}, \qquad  x <  -R.
%    \end{aligned}
%  \end{equation}
Then there exists $C_+,C_- >0 $ such that
\begin{equation}
  \label{eq:est fourier inversion}
    |I_{\infty}[\hat{h}](y)|\leq
    \begin{cases}
      C_+ y^{\gamma^+ - 1}e^{-by} ,\qquad y>  1,\\
      C_-|y|^{\gamma^- -1}e^{-a|y|} ,\qquad y<  -1.
    \end{cases}
\end{equation}
\end{theorem}
\begin{proof}
See \cref{app:teo:estimate fat tails}. \ref{sup:residueTheorem} contains a particular case of this theorem where the considered singularities are just poles.
\end{proof}

\begin{corollary}\label{cor: estimate fat tails truncated}
    Let $h$ as in \cref{teo:estimate fat tails} and $r>0$ satisfying the assumptions of \cref{lem:fourier box value}. Then we can find $C_+, C_->0$ such that
  \begin{equation}
  \label{eq:est truncated fourier inversion}
    |I_{r}[\hat{h}](y)|\leq\left\{
    \begin{aligned}
      &C_+y^{\gamma^+ - 1}e^{-by} + \frac{1}{\pi y} \left|\Im\left[ e^{\ii r  y} \hat{h} (r)\right]\right| + O(y^{-2}) ,\qquad y>  1,\\
      &C_-|y|^{\gamma^- -1}e^{-a|y|} + \frac{1}{\pi |y|} \left|\Im\left[ e^{\ii r  y} \hat{h} (r)\right]\right| + O(y^{-2}) ,\qquad y<  -1.
    \end{aligned} \right.
\end{equation}
\end{corollary}
\begin{proof}
The result follows immediately by considering $$|I_{r}[\hat{h}](y)|\leq |I_{\infty}[\hat{h}](y)|+\left|\frac{1}{2\pi}\int_{|u|\geq  r} \hat{h}(u)e^{\ii uy}\dd u\right|,$$
and by applying \cref{teo:estimate fat tails} and
\cref{lem:fourier box value} with $N=0$.
\end{proof}

% \begin{remark} \label{remark:Bounds}
%   If $ h $ satisfies the hypothesis of \cref{teo:estimate fat tails}, we have that there is $ r_0 $ such that, for $ r>r_0 $
%   $$ \left|\frac{1}{2\pi}\int_{-r}^{r} \hat{h}(u)e^{\ii uy}\dd u\right| \leq  \left|\frac{1}{2\pi}\int_{-\infty}^{\infty} \hat{h}(u)e^{\ii uy}\dd u\right| + \left|\frac{1}{2\pi}\int_{|u|\geq  r} \hat{h}(u)e^{\ii uy}\dd u\right|. $$
%   By \eqref{eq:est fourier inversion} and \cref{lem:fourier box value}, we have that
%   \begin{equation}
%   \label{eq:est truncated fourier inversion}
%     \left|\frac{1}{2\pi}\int_{-r}^{r} \hat{h}(u)e^{\ii uy}\dd u\right|\leq
%     \begin{cases}
%       C_Dy^{\gamma^+ - 1}e^{-by} + y^{-1}\varepsilon(r,y) ,\qquad y>  1,\\
%       C_D|y|^{\gamma^- -1}e^{-a|y|} + |y|^{-1}\varepsilon(r,y) ,\qquad y<  -1,
%     \end{cases}
% \end{equation}
%   where $ \varepsilon(r,y) $ is defined following Equation \eqref{eq:fourier box value} as 
%   \begin{equation}
%     \label{eq:def epsilon_r}
%     \varepsilon(r,y)=\frac{1}{\pi} \sum_{n=0}^{\infty}  \left| \Im\left[e^{-\ii yr} \hat{h}^{(n)}(-r)\right]\ii^{n} y^{-n}\right| =  \frac{1}{\pi}\left|\Im\left[e^{-\ii yr} \hat{h}(-r)\right]\right| + o(|y|^{-1}).
%   \end{equation}
% \end{remark}

In the following corollary, we apply \cref{cor: estimate fat tails truncated} to the Variance Gamma (VG) distribution. See \ref{sup:GH_NIG} for the Generalized Hyperbolic and Normal Inverse Gaussian.

\begin{corollary}
For the VG density (centered at $0$),
%\footnote{It is worth noting that for $0<\lambda<\frac{1}{2}$ this $\hat{f}\notin \mathbb{L}^1(\R)$ and it's not obvious how to invert $\hat{f}$. A possible set of assumptions is that $f$ is integrable, continuous, and with piecewise continuous derivative. If $0<\lambda <\frac{1}{2}$, and considered the PDF $f$, it is possible to show that $f$ has a singularity at $x=0$ that goes as $x^{2\lambda-1}$, and so it is integrable for all $\lambda>0$. However, this singularity is not square-integrable unless $\lambda>\frac{1}{4}$, in line with Plancherel's theorem. Accordingly, the VG is an example distribution with a blow-up close to zero, but for which the estimates \eqref{eq:est truncated fourier inversion} hold. There is no damping that can solve this problem. } 
%\footnote{The VG distribution can be obtained from the GH when $\delta \to 0$. Under this definition we have $a=\frac{1}{\sigma^2} \sqrt{\frac{2}{\nu}\sigma^2 + \theta^2}$, $b=\frac{\theta}{\sigma^2}$, $\gamma=\sqrt{\frac{2}{\nu \sigma^2}}$.}
\begin{equation*}
  %\label{eq:ft variance-gamma}
\hat{f}(u)=\left(1 +\ii\theta\nu u + \frac{1}{2}\sigma^2\nu u^2\right)^{-\lambda},\qquad \lambda=\frac{T}{\nu}>0.
\end{equation*}
%Notice that, the argument of $z^{-\lambda}$ is a second degree polynomial with $\Delta= -\theta^2\nu^2-2\sigma^2\nu <0$. Its first order coefficient is purely imaginary, while its zero term is real. Accordingly, we have that its roots are
The singularities are $u_1= \ii b$, $u_2=-\ii a$, where
$$a = \frac{\theta}{\sigma^2} + \sqrt{\frac{\theta^2}{\sigma^4}+\frac{2}{\sigma^2\nu}},\qquad b = -\frac{\theta}{\sigma^2} + \sqrt{\frac{\theta^2}{\sigma^4}+\frac{2}{\sigma^2\nu}}.$$
Therefore, on the real line, $\hat{f}(u)=[\frac12\sigma^2\nu(u-u_1)(u-u_2)]^{-\lambda}$, where the branch of $u^{-\lambda}$ is the usual one defined on $\C \setminus \{ u <0\}$. It is possible to select an extension of $\hat{f}$ defined in $\C\setminus \{\gamma_{u_1}, \gamma_{u_2}\}$ where the cuts are parallel to the imaginary axis. Considering such an extension, then $u_1$ and $u_2$ are algebraic singularities of order $\lambda$. Moreover, $\hat{f}$ decays uniformly at infinity. Therefore, by \cref{cor: estimate fat tails truncated}, we find $ C_+,C_->0 $ such that for all $y_k:=k/2^m\neq 0$, $r:=2^m\pi$
 \begin{equation*}
    |I_r[\hat{f}](y_k)|\leq
    \begin{cases}
      C_+ y_k^{\lambda-1}e^{-by_k}  +  \frac{1}{\pi y_k}\left|\Im[\hat{f}(2^m \pi)]\right| + O(y_k^{-2}),\quad y_k>  1,\\
      C_-|y_k|^{\lambda-1}e^{-a|y_k|}  + \frac{1}{\pi |y_k|}\left|\Im[\hat{f}(2^m \pi)]\right| + O(y_k^{-2}),\quad y_k<  -1.
    \end{cases}
\end{equation*}
\end{corollary}
\begin{remark}
	Experimentally, we can estimate  $C_+\approx D^{\alpha}_{m,k^*}/(y^{\gamma^{+}-1}e^{-(b-\alpha)y})$,
where 
	$ y= \frac{(\gamma^{+}-1)+ \sqrt{\gamma^{+}-1}}{b-\alpha} $ and 
$k^*/2^{m} - \mu \approx  \frac{(\gamma^{+}-1)+ \sqrt{\gamma^{+}-1}}{b-\alpha}$. This $ k^* $ and this value of $y$ identify the inflection point of $y^{\gamma^{+}-1}e^{-(b-\alpha)y}$. In case $ \gamma^{+}<1 $, then $C_+$ can be estimated using $y=1$ and $ k^*/2^{m} - \mu \approx 1$. A different constant should be estimated similarly for the left-tail, as the distribution could be asymmetric.
\end{remark}	
	\begin{corollary}\label{cor:SWIFT density bound}
		Let $f_0$ be a density function satisfying the assumptions in \cref{teo:estimate fat tails} and \cref{lem:fourier box value}. Suppose that $\hat{f}_0$ has nearest singularities with imaginary parts equal to $b>0$ and $-a<0$, and maximal orders $\gamma^+$ and $\gamma^-$.  Given $-a<\alpha<b$, $\mu \in \R$,  $y_k:=k/2^m$, $y_{k,\mu}:=y_k-\mu$, the Shannon coefficients for the damped-centered density $f_\alpha(y)=e^{\alpha y}f(y-\mu)$ verify
\begin{equation}
\label{eq:density_bound}
|D^\alpha_{m,k}|
\leq 2^{-m/2} \left\{
\begin{aligned}
& C_+\, e^{\alpha\mu} (y_{k,\mu})^{\gamma^+ -1}
e^{-(b-\alpha)y_{k,\mu}} + \varepsilon_D(m)y_{k,\mu}^{-1}, \quad y_{k,\mu} > 1,\\
%\quad + \varepsilon_D(m)\left|y_{k,\mu}\right|^{-1} + O\left(y_{k,\mu}^{-2}\right),
&C_-\, e^{\alpha\mu}\left|y_{k,\mu}\right|^{\gamma^- -1}
e^{-(a+\alpha)\left|y_{k,\mu}\right|} + \varepsilon_D(m)\left|y_{k,\mu}\right|^{-1}, \quad y_{k,\mu} < -1,
%& y_{k,\mu} < -1. \\
%\quad + \varepsilon_D(m)\left|y_{k,\mu}\right|^{-1}+O\left(y_{k,\mu}^{-2}\right),
\end{aligned}
\right.
\end{equation}
with $\varepsilon_D(m)=\frac{1}{\pi}|\Im[\hat{f}_\alpha(2^m\pi)]|=\frac{1}{\pi}|\Im[\hat{f}(2^m\pi+\ii\alpha)]|$. This estimation is precise up to an additional $O(y_{k,\mu}^{-2})$ term.
\end{corollary}
\begin{proof}
By the translation property of the FT, $\hat{f}_0(y-\mu)=\hat{f}_0(u) e^{-\ii \mu u}$. Since exponential damping leads to an evaluation at $u+ \ii \alpha$, we have
\begin{equation}
\label{eq:proof cor 3.12, dampped centered TF}
\hat{f}_\alpha(u)=e^{\alpha\mu}\hat{f}_0(u+
			\ii\alpha)e^{-\ii\mu u}.
\end{equation}
Then, any algebraic (logarithmic) branch point for $\hat{f_\alpha}$ is equal to
$u_1 - \ii\alpha , \ldots, u_N- \ii\alpha$
where $u_1,\ldots, u_N$ are the branch points of $\hat{f}_0$, and their orders do not vary. Notice that
$$I_r[\hat{f}_\alpha](y)=\frac{e^{\alpha \mu}}{2\pi}\int_{-r}^{r} \hat{f_0}(u+\ii\alpha) e^{\ii (y-\mu) u }\dd u.$$
Since $D^\alpha_{m,k} = 2^{-m/2}\,I_r[\hat{f_\alpha}](y_k)$, the first summand in \eqref{eq:density_bound} follows from \eqref{eq:proof cor 3.12, dampped centered TF} and \cref{teo:estimate fat tails} applied to $\hat{f}_0(u+\ii \alpha)$ with $y=y_{k,\mu}$ and $r=2^m \pi$, whereas the second is a direct consequence of \cref{lem:fourier box value} applied to $\hat{f}_\alpha$.
\end{proof}

The following result follows from our set of assumptions and \cref{lem:fourier box value}.
\begin{corollary}%[Damped Payoff]
\label{cor:SWIFT payoff bound}
    Suppose that the payoff $P$ is piecewise $C^1$, and that $P_\alpha$ satisfies the hypothesis of \cref{lem:fourier box value}. Then if $k\neq 0$ and $y_k$ is not a jump point for $P$, up to a negligible $O(y_k^{-2})$ term, one has
\begin{equation}
\label{eq:general_payoff_bound weak}
        |V^\alpha_{m,k}| \leq  2^{-m/2} \left(  |P_{\alpha}(y_k)|+\varepsilon_V(m) |y_k|^{-1} \right),
\end{equation}
with $\varepsilon_V(m)=\frac{1}{\pi}|\Im[\hat{P}_\alpha(2^m \pi )]|=\frac{1}{\pi}|\Im[\hat{P}(2^m \pi -\ii\alpha)]|$.
%Moreover, if $\hat{P}_\alpha \in L^1(\R)$, we also have the uniform bound
%\begin{equation}
%\label{eq:general_payoff_bound}
%        |V^\alpha_{m,k}| \leq  2^{-m/2} \left(  |P_{\alpha}(y_k)| + \zeta_k \right),\qquad \zeta_k := \min\left(\epsilon_r(m),  \varepsilon_V(m) |y_k|^{-1} \right)
%\end{equation}
%where $ \epsilon_r(m):=\frac{1}{2\pi}\int_{|u|\geq 2^m\pi} |\hat{P}_\alpha(u)|\dd u$.
\end{corollary}
\begin{proof}
Given that $P_\alpha\in L^1(\R)$ is piecewise $C^1$ and $y_k$ is not a jump point, by the Fourier inversion formula we obtain that
\begin{equation*}
    |V^\alpha_{m,k}| \leq 2^{-m/2}\left( |P_\alpha(y_k)| + \left| \frac{1}{2\pi}\int_{|u|\geq 2^m\pi} \hat{P}_\alpha(u)e^{\ii u k /2^m}\dd u \right|\right).
\end{equation*}
Therefore, the bound \eqref{eq:general_payoff_bound weak} immediately follows from \cref{lem:fourier box value}.
%, and if $\hat{P}_\alpha \in L^1(\R)$ we also have the uniform bound
%   $$\left| \frac{1}{2\pi}\int_{|u|\geq 2^m\pi} \hat{P}_\alpha(u)e^{\ii u k /2^m}\dd u \right| \leq \frac{1}{2\pi}\int_{|u|\geq 2^m\pi} |\hat{P}_\alpha(u)|\dd u ,  $$
%and this implies \eqref{eq:general_payoff_bound}.
\end{proof}

\section{Error decomposition and parameter selection}\label{Error Decomposition}
In this section, we aim to study the error of the damped SWIFT method. In our setting, as Assumptions \ref{ass: f and P regularity. P jumps} and \ref{ass:noempty strip intersection} hold, the exact European option value $v$ can be approximated by the coefficient representation of order $m$, denoted as $v_m$, and defined in Equation \eqref{eq:v_m(x)}.
%defined as the inner product between the coefficient representation of the damped density and that of the damped payoff projected onto the Shannon wavelet space of order $m$. 

The damped SWIFT approximation is then obtained in two steps. First, the representation is truncated to a finite set $\Lambda=[l,u]\cap\Z$, yielding $v_{m,\Lambda}$ and defined in Equation \eqref{eq:v_m,Lambda(x)}. Then, the resulting coefficients are approximated by numerical quadrature, yielding $v_{m,\Lambda,Q}$, as in Equation \eqref{eq:v_m,Lambda,Q(x)}.  %Consequently, the total approximation error naturally decomposes into three contributions: the projection error , the truncation error , and the quadrature error, arising from $v-v_m$, $v_m-v_{m,\Lambda}$, and $v_{m,\Lambda} - v_{m,\Lambda,Q} $, respectively.
%Throughout this section, we use the approximations introduced in \cref{sec:DSWIFT}, namely $v_m$, $v_{m,\Lambda}$ and $v_{m,\Lambda,Q}$, with $\Lambda = [l,u]\cap\Z$, $l<u$. The coefficient representation of \cref{sec:DSWIFT} only requires the admissible damping condition.  Some of the estimates below require additional assumptions, which are stated locally.  In particular, the projection-error bound uses $\hat{f}_\alpha(\cdot\mid x)\in L^1(\R)$, while some payoff-coefficient tail estimates use $\hat{P}_\alpha\in L^1(\R)$.  These stronger Fourier-side assumptions are not imposed globally, since discontinuous payoffs such as digital options may fail to satisfy $\hat{P}_\alpha\in L^1(\R)$.
By the triangular inequality, we obtain that the error of the method can be naturally decomposed as
 \begin{align} \label{eq:error_decompositon}
e^{r_0T}|v(x)-v_{m,\Lambda,Q}(x)|
\le &\, E_{\rm proj}(m,\alpha)
   +E_{\rm trunc}(m,\Lambda,\alpha) + E_{{\rm{quad}},}(m,\Lambda,Q,\alpha),
   %\\
  % &E_{{\rm{quad}},D}(m,\Lambda,Q,\alpha)  +E_{{\rm{quad}},V}(m,\Lambda,Q,\alpha),\nonumber
 \end{align}
where
 \begin{equation}
 \label{eq:error_decompositon definition}
\begin{gathered}
      E_{\rm proj}(m,\alpha)
:= e^{r_0 T}| v(x)- v_m(x)|, \qquad E_{\rm trunc}(m,\Lambda,\alpha),
:= e^{r_0 T}|v_m(x) - v_{m,\Lambda}(x)| \\ E_{{\rm quad}}(m,\Lambda,Q,\alpha)
:= e^{r_0 T} |v_{m,\Lambda}(x) - v_{m,\Lambda ,Q}(x)|.
     \end{gathered}
 \end{equation}
%\begin{align*}
%E_{\rm proj}(m,\alpha)
%&:=
%\left|
%\int_\R P_\alpha(y)
%\left(f_\alpha(y\mid x)-\sum_{k\in\Z}D^\alpha_{m,k}(x)\varphi_{m,k}(y)\right)\dd y
%\right|,\\
%E_{\rm trunc}(m,\Lambda,\alpha)
%&:=
%\left|
%\sum_{k\in\Z\setminus\Lambda}D^\alpha_{m,k}(x)V^\alpha_{m,k}
%\right|,\\
%E_{{\rm quad},D}(m,\Lambda,Q,\alpha)
%&:=
%\left|
%\sum_{k\in\Lambda}V^\alpha_{m,k}
%\big(D^\alpha_{m,k}(x)-D^{\alpha,Q}_{m,k}(x)\big)
%\right|,\\
%E_{{\rm quad},V}(m,\Lambda,Q,\alpha)
%&:=
%\left|
%\sum_{k\in\Lambda}D^{\alpha,Q}_{m,k}(x)
%\big(V^\alpha_{m,k}-V^{\alpha,Q}_{m,k}\big)
%\right|.
%\end{align*}
%and notice that this decomposition separates projection, translation-index truncation, and coef\-fi\-cient-quadrature errors. 
The aim of this section is to provide bounds for each component. To achieve this, we require an additional assumption.
\begin{assumption}
\label{ass:error analysis}
    In this section, we assume that $\hat{f}_\alpha \in L^1(\R)$ in order to bound the projection error.
    %We also assume that $\hat{P}_\alpha\in L^1(\R)$ to bound the payoff coefficients $V^\alpha_{m,k}$.
\end{assumption}
%The error analysis in \cite{OrtizOosterlee2016,colldeforns2017two} assumes the boundedness of the payoff, which we do not assume here. This comes with the advantage that our approach allows the truncation range to be informed by the behavior of the payoff coefficients as well (as explained in \cref{sec:truncError}), rather than solely by the density coefficients, as in the seminal works. 
%The numerical illustrations in Section \ref{sec:NUmerical experiments} further support this. 
%On the other hand, these articles incur an additional domain truncation error, as the computations for the density and payoff coefficients are performed in the physical space.

\subsection{Projection error} 
\Cref{ass:error analysis} for $\hat{f}_\alpha$ can be applied to obtain the following bound.
%for the projection error.
\begin{proposition}
%[Projection error]
\label{prop:projection-error}
Assume that $P_\alpha\in L^1(\R)$ and $\hat{f}_\alpha(\cdot\mid x)\in L^1(\R)$.  Then
\begin{equation}\label{eq:projection-error-bound}
E_{\rm proj}(m,\alpha)
\le
\|P_\alpha\|_{L^1(\R)}
\frac1{2\pi}
\int_{|w|>2^m\pi}|\hat{f}_\alpha(u \mid x)|\,\dd u.
\end{equation}
\end{proposition}

\begin{proof}
Under the stated Fourier-inversion assumptions,
\[
f_\alpha(y\mid x)-\sum_{k\in\Z}D^\alpha_{m,k}(x)\phi_{m,k}(y)
=
\frac1{2\pi}\int_{|u|>2^m\pi}\hat{f}_\alpha(u\mid x)e^{\ii u y}\,\dd u.
\]
Using this representation in $E_{\rm proj}(m,\alpha)$, and taking absolute values, gives \eqref{eq:projection-error-bound}.
\end{proof}

The bound \eqref{eq:projection-error-bound} can be used to select the resolution level $m$.  %In practice one may also use cheaper diagnostics based on the magnitude of the characteristic function at the endpoints $\pm 2^m\pi$, but such diagnostics should be interpreted as heuristics rather than rigorous upper bounds for the full Fourier tail.

\subsection{Truncation Error} \label{sec:truncError}
We now study the error component $E_{\rm trunc}(m,\Lambda,\alpha)$ due to the truncation of the series in \eqref{eq:v_m(x)} to a sum where the translation parameter $k\in \Lambda =[l,u]\cap \Z$, with $l<u$. 
We denote the product coefficient as $\mathcal{P}_{m,k}^{\alpha}(x) := D^{\alpha}_{m,k}(x)\cdot V^{\alpha}_{m,k}$ and notice that
\begin{equation}\label{eq:trunc_error}
      E_{\rm trunc}(m,\Lambda,\alpha)
      %= \Bigl|\sum_{k \notin [l,u]} D^{\alpha}_{m,k}\,V^{\alpha}_{m,k}\Bigr| 
      \leq \sum_{k < l} |\mathcal{P}_{m,k}^{\alpha}| + \sum_{k > u} |\mathcal{P}_{m,k}^{\alpha}| := \bar{T}^{-}(l) + \bar{T}^{+}(u).
\end{equation}
The results given in the following \cref{teo:precise truncation error bound} constitute quite precise bounds that have been applied in \cref{sec:numREsults}, while \cref{cor: estimate fat tails truncated} clearly shows that the truncation error decays at first exponentially, and then algebraically as $|u|$ and $|l|$ tend to infinity.
In the following, we apply the results in \cref{sec:Asymptotic decay of the payoff and density coefficients with respect to the translation parameters} along with the following lemma.

\begin{lemma}
%[Tail-sum bound for piecewise-monotone envelopes]
\label{lem:tail-sum-envelope}
Let $G:(0,\infty)\to[0,\infty)$ be integrable, piecewise monotone, and with finitely many local maxima $y_1^*,\ldots,y_N^*$. Then, for $U>0$ we have
\begin{equation}\label{eq:tail-sum-envelope}
2^{-m}\sum_{\frac{k}{2^{m}}>U}G(k/2^m)
\le
\int_U^\infty G(y)\dd y
+2^{1-m}\sum_{n=1}^N G(y_n^*)\mathbf 1_{\{U<y_n^*\}}.
\end{equation}
\end{lemma}
\begin{proof}
$G$ is piecewise monotone in $[U,\infty)$ and tends to zero at infinity, as it is integrable. It follows that $G$ has a bounded variation. By the Ostrowski-integral inequality (see \cite[Theorem 3]{dragomir1999ostrowski}), it holds that
$$\left|\int_U^\infty G(y) \dd y-2^{-m}\sum_{\frac{k}{2^{m}}>U}G(k/2^m)\right| \leq 2^{-m}\mathbf{V}(G),
$$
where $\mathbf{V}(G)$ denotes the variation of $G$ in $[U,\infty)$. $\mathbf{V}(G)$ is equal to the sum of the variations on each interval where the function is monotone. Namely, we have that $\mathbf{V}(G) = 2 \sum_{U < y^*_n} (G(y^*_n) - G(y_{n,*})),$
where $y_{n,*}$ is a local minimum of $G$. Since $G$ is positive, $\mathbf{V}(G) \leq  2\sum_{U < y^*_n} G(y^*_n)$, and this gives \eqref{eq:tail-sum-envelope}.
\end{proof}

We now prove the main result related to the truncation error.
\begin{theorem}
\label{teo:precise truncation error bound}
Let $f$ be a density function and $P$ a payoff function satisfying Assumptions \ref{ass: f and P regularity. P jumps} and \ref{ass:noempty strip intersection}. Let $\alpha \in \delta_V$ such that $P_\alpha$ and $f_\alpha$ verify the assumptions in \cref{lem:fourier box value} and \cref{teo:estimate fat tails}. Let $\Gamma(s,y) = \int_y^\infty t^{s-1}e^{-t}\dd t$ be the upper incomplete Gamma function. Given $\mu\in \R$ then for $U:=\frac{u}{2^m}>\max(0, 1+\mu)$, up to $O(U^{-2})$, we have
\begin{align}
	\bar{T}^{+}(u) \leq &  \, C_+\, e^{\alpha \mu}\frac{ \|P_{\alpha}\|_{L^\infty([U,\infty))} + \varepsilon_V U^{-1}}{(b-\alpha)^{\gamma^+}} \Gamma(\gamma^+, (b-\alpha)(U-\mu)) \label{eq:final tail bound right precise}\\
	& + \varepsilon_D \left\| P_{\alpha} \right\|_{L^1([U,\infty))} (U-\mu)^{-1}  + \varepsilon_D \varepsilon_V \mu^{-1}\log\left(\frac{U}{U-\mu}\right) \nonumber \\
    &+ 2^{1-m}\sum_{n=1}^N G_+(y_n^*)\mathbf 1_{\{U<y_n^*\}},
    \nonumber
\end{align}
where $y^*_n$ denotes the local maxima in $(0,\infty) \cap (1+\mu,\infty)$ of 
$$G_+(y)= \left(C_+\, e^{\alpha \mu} (y-\mu)^{\gamma^+-1}
 e^{-(b-\alpha)(y-\mu)} +   \varepsilon_D (y-\mu)^{-1}\right) \left(|P_{\alpha}(y)| + \varepsilon_V |y|^{-1}\right) .$$
Analogously, if $L:=\frac{l}{2^m}<\min(0,\mu -1)$ then, up to $O(L^{-2})$, it holds that
\begin{align}
	\bar{T}^{-}(l) \leq &  \, C_- \, e^{\alpha \mu} \frac{ \|P_{\alpha}\|_{L^\infty((-\infty,L])} + \varepsilon_V |L|^{-1}}{(a+\alpha)^{\gamma^-}} \Gamma(\gamma^{-},(a+\alpha)|L-\mu|) \label{eq:final tail bound left precise}  \\
	& + \varepsilon_D \left\| P_{\alpha} \right\|_{L^1((-\infty,L])} |L-\mu|^{-1}  + \varepsilon_D \varepsilon_V \mu^{-1}\log\left(\frac{L-\mu}{L}\right)\nonumber\\%\log(|1-\mu L^{-1}|) \nonumber\\
    & + 2^{1-m}\sum_{n=1}^N G_-(y_n^*)\mathbf 1_{\{L>y_n^*\}},\nonumber
\end{align}
where $y^*_n$ denotes the local maxima in $(-\infty,0) \cap (-\infty, \mu -1)$ of 
$$G_-(y)= \left(C_- \, e^{\alpha \mu}  |y-\mu|^{\gamma^+-1}
 e^{-(a+\alpha)|y-\mu|} +   \varepsilon_D |y-\mu|^{-1}\right) \left(|P_{\alpha}(y)| + \varepsilon_V |y|^{-1}\right) .$$
\end{theorem}
\begin{proof}
We prove the result for $\bar{T}^{+}(u) $, as the proof for $\bar{T}^{-}(l) $ is analogous. Given the assumptions on $P_\alpha$ and $f_\alpha$, Corollaries \ref{cor:SWIFT density bound} and \ref{cor:SWIFT payoff bound} hold, and we can apply the bounds in Equations \eqref{eq:density_bound} and \eqref{eq:general_payoff_bound weak}. Namely, when $ y_k=\frac{k}{2^m} >\mu + 1 $, we have \footnote{ In case $y_k$ corresponds to a jump point of the payoff, Equation \eqref{eq:general_payoff_bound weak} remains valid with $P(y_k)$ replaced by the average of the left and right limits of the payoff at $y_k$. The maximum between the left limit $P(y_k^-)$ and the right limit $P(y_k^+)$ can be used in the bound of $|\mathcal{P}_{m,k}^{\alpha}(x)|$. }
    \begin{align}
	\left|\mathcal{P}_{m,k}^{\alpha}(x)\right| \leq& 
    2^{-m} C_+\, e^{\alpha \mu} \left(|P_{\alpha}(y_k)| + \varepsilon_V |y_k|^{-1}|\right) (y_{k}-\mu)^{\gamma^+-1}
 e^{-(b-\alpha)y_{k}-\mu}  \label{eq:overall product bound with mu} \\
  & \hspace{-0.15cm}+ 2^{-m}   \varepsilon_D (y_{k}-\mu)^{-1} \left(|P_{\alpha}(y_k)| + \varepsilon_V \, |y_{k}|^{-1} \right) = 2^{-m} \, G_+(y_{k}). \nonumber 
\end{align}
By assumption, $y_k>0$ and $y_k >1+\mu$, hence $G_+$ is defined in $(0,\infty)\cap (1+\mu, \infty)$.
 By \cref{ass:noempty strip intersection}, $ P_{\alpha}\in L^1(\R)\cap L^\infty(\R) $, and by \cref{ass: f and P regularity. P jumps}, the damped payoff is piecewise monotone with finitely many critical points. The same holds for the functions
 $$y_{k} \mapsto  (y_{k}-\mu)^{\gamma^+-1}
 e^{-(b-\alpha)(y_{k}-\mu)}, \quad \text{and} \quad y_{k} \mapsto  \varepsilon_V|y_{k}|^{-1},$$
 defined in $(0,\infty)\cap (1+\mu, \infty)$.
 In light of this, the function $G_+$, which provides an upper bound for the product coefficients, is piecewise monotone, has finitely many local maxima $y^*_1, \ldots, y^*_N$, and is integrable. Therefore, \cref{lem:tail-sum-envelope} applies, yielding
 \begin{equation*}
    \bar{T}^{+}(u)\leq  \int_U^\infty G_+(y)\dd y
+2^{1-m}\sum_{n=1}^N G_+(y_n^*)\mathbf 1_{\{U<y_n^*\}}.
 \end{equation*}
The statement given by Equation \eqref{eq:final tail bound right precise} follows from the estimation of the terms that define $G_+$. Namely,
\begin{equation*}
    \begin{aligned}
       \int_{U}^\infty G_+(y)\dd y =& C_+ \, e^{\alpha \mu}\int_{U}^\infty \left(|P_{\alpha}(y)| + \varepsilon_V |y|^{-1}\right) (y-\mu)^{\gamma^+-1} e^{-(b-\alpha)(y-\mu)} \dd y \\
       &  + \varepsilon_D \int_{U}^{\infty} \left(|P_{\alpha}(y)| +\varepsilon_V |y|^{-1}\right)(y-\mu)^{-1} \dd y \\
       \leq &  C_+ \, e^{\alpha \mu}\left(\| P_{\alpha}\|_{L^\infty([U,\infty))} +\varepsilon_V U^{-1}\right) \int_U^\infty  (y-\mu)^{\gamma^+-1}e^{-(b-\alpha)(y-\mu)} \dd y \\
       & +  \varepsilon_D (U-\mu)^{-1} \int_{U}^{\infty} |P_{\alpha}(y)| \dd y + \varepsilon_D \varepsilon_V \int_U^\infty |y|^{-1}(y-\mu)^{-1} \dd y\\
        = &C_+ \, e^{\alpha \mu}\left(\| P_{\alpha}\|_{L^\infty([U,\infty))} +\varepsilon_V U^{-1}\right) (b-\alpha)^{-\gamma^+}\Gamma(\gamma^+, (b-\alpha)(U-\mu)) \\
       & + \varepsilon_D \left\| P_{\alpha} \right\|_{L^1([U,\infty))} (U-\mu)^{-1} + \varepsilon_D \varepsilon_V \mu^{-1}\log\left(\frac{U}{U-\mu}\right), %- \varepsilon_D \varepsilon_V \mu^{-1}\log(|1-\mu U^{-1}|),
    \end{aligned}
\end{equation*}
and this concludes the proof.
\end{proof}
\begin{corollary}
Consider a density $f$ and a payoff $P$, along with $\alpha\in \R$ as in \cref{teo:precise truncation error bound}. If $U=\frac{u}{2^m}>\max(0,\mu+1,\mu +\frac{\gamma^+-1}{b-\alpha}, y^*_N)$ then, up to $O(U^{-2})$, we have
    \begin{align}
	\bar{T}^{+}(u) \leq &  \, C_+ \, e^{\alpha\mu} \left(\|P_{\alpha}\|_{L^\infty([U,\infty))} + \varepsilon_V U^{-1}\right) \frac{(U-\mu)^{\gamma^{+}-1}e^{-(b-\alpha)(U-\mu)}}{(b-\alpha)- \frac{\gamma^{+}-1}{U-\mu}} \label{eq:final tail bound right}\\
&+ \varepsilon_D  (\left\| P_{\alpha} \right\|_{L^1([U,\infty)}  + \varepsilon_V) \max\left\{U^{-1},(U-\mu)^{-1}\right\}. \nonumber %+ 2^{1-m}\sum_{n=1}^N G_+(y_n^*)\mathbf 1_{\{U<y_n^*\}},\nonumber
\end{align}
%   In general, for $U$ sufficiently big we have
%   \begin{equation}
%   \label{eq:final tail bound right big U}
%	\hat{T}^{+}(u) \leq %&  \, C_+ e^{\alpha \mu} \frac{ \|P_{\alpha}\|_{L^\infty([U,\infty))} + \epsilon_r}{(b-\alpha)- (\gamma^{+}-1)(U-\mu)^{-1}} (U-\mu)^{\gamma^{+}-1}e^{-(b-\alpha)(U-\mu)}  + \nonumber\\
%	 \varepsilon_D \left\| P_{\alpha} \right\|_{L^1([U,\infty))} (U-\mu)^{-1}  + \varepsilon_D \varepsilon_V U^{-1}.
%\end{equation}
An analogous bound holds for the left tail $\bar{T}^-(l)$. In particular, the truncation error of the damped SWIFT method exhibits two different decaying regimes. Under the first one, the bound is $\lesssim U^{\gamma^+-1}e^{-(b-\alpha)U}$ ($|L|^{\gamma^--1}e^{-(a+\alpha)|L|}$), whose decaying is exponential, while it is $\lesssim  U^{-1}$ ($|L|^{-1}$) under the second regime.
\end{corollary} 
\begin{proof}
   The statement given by Equation \eqref{eq:final tail bound right} follows by estimating the terms in \eqref{eq:final tail bound right precise}. Under the assumption $U>y^*_N$, the terms containing $G_+$ in \eqref{eq:final tail bound right precise} vanish. Using integration by parts, the upper incomplete Gamma function satisfies
    \begin{align*}
      \int_U^{\infty} y^{a}e^{-by}\dd y &= b^{-1} U ^{a}e^{-bU} + b^{-1} a\int_U^{\infty} y^{a-1}e^{-by}\dd y \\
      &\leq   b^{-1} U ^{a}e^{-bU} + b^{-1} a U^{-1}\int_U^{\infty} y^{a}e^{-by}\dd y,
    \end{align*}
%Accordingly, 
%$$ \left( 1-b^{-1}a U^{-1}\right) \int_U^{\infty} x^{a}e^{-bx}dx \leq  b^{-1} U^{a}e^{-bU} ,$$
and if $ b-aU^{-1} >0 $ then $\int_U^{\infty} y^{a}e^{-by} \dd y\leq  \frac{ U^{a} e^{-bU} }{b-aU^{-1}}$. 
In our case, this condition is equivalent to $U > \mu + \frac{\gamma^+-1}{b-\alpha}$. For the logarithmic term in \eqref{eq:final tail bound right precise}, we apply
\begin{equation*}
%\label{eq:proof cor tail bound, log bound}
    \frac{1}{\mu}\log\left(\frac{U}{U-\mu}\right)\leq \max\left(\frac{1}{U}, \frac{1}{U-\mu}\right),
\end{equation*}
which is valid for $U>0$ and $U>\mu$, and this gives us \eqref{eq:final tail bound right}.
%In fact, for $\mu >0$, the bound \eqref{eq:proof cor tail bound, log bound} is equivalent to
%\begin{equation*}
%    \begin{aligned}
%    &\log\left(\frac{U}{U-\mu}\right) \leq \frac{\mu}{U-\mu} \iff
%    \log\Big{(}1 +\underbrace{\frac{\mu}{U-\mu}}_{=t}\Big{)} \leq \underbrace{\frac{\mu}{U-\mu}}_{=t}
%    \end{aligned}
%\end{equation*}
%which is true since $\log(1+t)\leq t$ when $t\geq 0$. On the other hand, when $\mu<0$, the bound is equivalent to 
%\begin{equation*}
%    \begin{aligned}
%    &\log\left(\frac{U}{U-\mu}\right) \geq \frac{\mu}{U} \iff
%    \log\Big{(}1 -\frac{\mu}{U}\Big{)} \leq - \frac{\mu}{U},
%    \end{aligned}
%\end{equation*}
%and since $-\frac{\mu}{U}>0$, the estimation \eqref{eq:proof cor tail bound, log bound} follows as before. 
\end{proof}
\begin{remark}[Support simplifications for vanilla payoffs]\label{rem:support-simplification}
For a put payoff $P(y)=(K-e^y)^+$, $P_\alpha(y)=0$ for $y\ge \log K$. Therefore, the right-tail payoff norm terms vanish once $U\ge \log K$.  For a call payoff $P(y)=(e^y-K)^+$, $P_\alpha(y)=0$ for $y\le \log K$, and the corresponding left-tail payoff terms vanish once $L\le \log K$.
\end{remark}

\subsection{Quadrature errors}
	\label{sec:quadErrors}
	We now bound the error produced by approximating the coefficient integrals
	\eqref{eq:D-fourier-real}--\eqref{eq:V-fourier} by the uniform composite
	trapezoidal rule. Let 
    %$c_m:=\frac{2^{-m/2}}{2\pi}$ and 
    \(N_Q\) be the number of trapezoidal nodes in
	\([-2^m\pi,2^m\pi]\), with $
	h_m:=%\frac{2r_m}{N_Q-1}
	%=
	\frac{2^{m+1}\pi}{N_Q-1}$. 
% For a sufficiently regular function \(g\), we have 
% 	\[
% 	\mathcal Q_{N_Q}[g]
% 	:=
% 	h_m\left[
% 	\frac12 g(-r_m)
% 	+
% 	\sum_{j=1}^{N_Q-2} g(-r_m+jh_m)
% 	+
% 	\frac12 g(r_m)
% 	\right].
% 	\]
	For \(k\in\Lambda\), set $
		\Delta^D_{m,k}(x)
		:=
		D^\alpha_{m,k}(x)-D^{\alpha,Q}_{m,k}(x),\;$ and  
	$	\Delta^V_{m,k}
		:=
		V^\alpha_{m,k}-V^{\alpha,Q}_{m,k}$.
	Then
	\begin{equation*}
		%\label{eq:quadrature-error-bound-basic}
		%\frac{|v_{m,\Lambda}(x)-v_{m,\Lambda,Q}(x)|}{e^{-r_0T}}
        E_{{\rm quad}}(m,\Lambda,Q,\alpha)
		\le
		E_{{\rm quad},D}(m,\Lambda,Q,\alpha)
		+
		E_{{\rm quad},V}(m,\Lambda,Q,\alpha),
	\end{equation*}
	where
	\begin{align}
		E_{{\rm quad},D}(m,\Lambda,Q,\alpha)
		&:=
		\sum_{k\in\Lambda}|V^\alpha_{m,k}|\,|\Delta^D_{m,k}(x)|, \qquad |\Delta^D_{m,k}(x)|=O(h_m^2),
		\label{eq:quadD-bound}\\
		E_{{\rm quad},V}(m,\Lambda,Q,\alpha)
		&:=
		\sum_{k\in\Lambda}|D^{\alpha,Q}_{m,k}(x)|\,|\Delta^V_{m,k}|,\qquad |\Delta^V_{m,k}|=O(h_m^2).
		\label{eq:quadV-bound}
	\end{align}
% 	Equivalently, if coefficientwise quadrature tolerances satisfy
% 	\[
% 	|\Delta^D_{m,k}(x)|\le \eta^D_{m,k}(N_Q),
% 	\qquad
% 	|\Delta^V_{m,k}|\le \eta^V_{m,k}(N_Q),
% 	\qquad k\in\Lambda,
% 	\]
% 	then
% 	\begin{align}
% 		E_{{\rm quad},D}(m,\Lambda,N_Q,\alpha)
% 		&\le
% 		\sum_{k\in\Lambda}|V^\alpha_{m,k}|\,\eta^D_{m,k}(N_Q),
% 		\label{eq:quadD-bound-tolerance}\\
% 		E_{{\rm quad},V}(m,\Lambda,N_Q,\alpha)
% 		&\le
% 		\sum_{k\in\Lambda}|D^{\alpha,Q}_{m,k}(x)|\,\eta^V_{m,k}(N_Q).
% 		\label{eq:quadV-bound-tolerance}
% 	\end{align}
% 	The exact coefficients appearing on the right-hand sides can be bounded, where the
% 	corresponding assumptions are satisfied, by
% 	\cref{cor:SWIFT payoff bound,cor:SWIFT density bound}. In particular,
% 	\[
% 	|D^{\alpha,Q}_{m,k}(x)|
% 	\le
% 	|D^\alpha_{m,k}(x)|+\eta^D_{m,k}(N_Q).
% 	\]
% We next give a finite-interval a priori bound for the coefficientwise quadrature
% errors. Define
% \[
% G^{D}_{m,k,\alpha}(u;x)
% :=
% \widehat f(u+\ii\alpha\mid x)e^{\ii k u/2^m},
% \qquad
% G^{V}_{m,k,\alpha}(u)
% :=
% \widehat P(u-\ii\alpha)e^{\ii k u/2^m}.
% \]
% In the models considered in this paper, these integrands are typically analytic in a
% neighbourhood of the real integration interval, provided that the shifted contours
% \(u+\ii\alpha\) and \(u-\ii\alpha\) do not intersect the singularity sets of
% \(\widehat f\) and \(\widehat P\), respectively. For the second-order finite-interval
% trapezoidal estimate, however, it is sufficient to assume the weaker condition
% \[
% G^{D}_{m,k,\alpha}(\cdot;x)\in C^2([-r_m,r_m]),
% \qquad
% G^{V}_{m,k,\alpha}\in C^2([-r_m,r_m]).
% \]
In the estimations of $|\Delta^D_{m,k}(x)|$ and $|\Delta^V_{m,k}|$, the standard composite trapezoidal estimate on a finite interval \cite{DavisRabinowitz1984} was considered.

\subsection{Damping-parameter selection}\label{subsec:damping-selection}
The damping parameter influences the size, oscillation, and regularity of the Fourier integrands used for the coefficient computation. We use the following practical criterion to select $\alpha$ within the admissible set:
\begin{equation}\label{eq:damping-objective}
\alpha^*
\in
\operatorname*{arg\,min}_{\alpha\in\delta_V(x)}
J(\alpha),
\qquad
J(\alpha):=|\hat{P}_\alpha(0)|\,|\hat{f}_\alpha(0\mid x)|.
\end{equation}
For nonnegative payoff and density factors, this objective is simply the product of their $L^1$ masses.  The purpose of \eqref{eq:damping-objective} is to reduce the peak size of the coefficient integrands and thereby improve quadrature conditioning.  Under the positivity and ridge assumptions used in the optimal-damping Fourier-pricing literature \cite{BayerBenHammoudaPapapantoleonSametTempone2022,bayer2024quasi}, minimizing the peak at the origin is equivalent to minimizing the sup-norm of the damped integrand on the integration contour.

\subsection{Algorithm} \label{sec:algorithm}
%The strategy is as follows. We begin by selecting the value of the damping parameter $\alpha$. Then, given a target accuracy $\epsilon$, we first choose the smallest integer $m$ such that the projection error is below $\epsilon$ (we want $m$ to be as small as possible, since the larger $m$ is, the more coefficients become significant). Next, we start from $k^* = \lfloor 2^m \mu  \rfloor$, and we proceed to the right until we find the smallest $|u|$ such that $\hat{T}^{+}(u)$ is of the same order as the projection error. We then proceed analogously to the left until finding the smallest $|l|$. To make the choice of $l$ and $u$ as tight as possible, the integrals in \eqref{eq:tailA}, \eqref{eq:tailB}, \eqref{eq:tailC}, and \eqref{eq:tailD} can be computed numerically, rather than using the proposed bounds. Finally, we choose the smallest number of quadrature nodes such that the quadrature errors are of the same order as the projection error.
The procedure for selecting the parameters of the method is shown in \cref{alg:dampedSWIFT}. Given a target accuracy, we first choose the smallest resolution level such that the projection error is below the prescribed tolerance. 
%Given a target accuracy $\epsilon$, we first choose the smallest integer $m$ such that the projection error is below $\epsilon$ (we want $m$ to be as small as possible, since the larger $m$ is, the more coefficients become significant). 
Finally, $l$, $u$ and $N_Q$ are the smallest integers such that truncation and quadrature errors are smaller than the product of the projection error and a parameter $\Upsilon\in(0,1)$.

\begin{algorithm}
\caption{Parameter selection for Damped SWIFT}
\label{alg:dampedSWIFT}
\begin{algorithmic}[1]
\STATE{\textbf{Inputs:} target accuracy, strip $\delta_V$, Fourier transforms, ratio $\Upsilon\in(0,1)$.}
\STATE{Select $\alpha\in\delta_V(x)$ by minimizing \eqref{eq:damping-objective}.}
\STATE{Choose the smallest resolution level $m$ such that the numerical approximation of \eqref{eq:projection-error-bound}, denoted as $\bar{E}_{\rm proj}(m,\alpha)$, is below the target accuracy.}
\STATE{Starting from $k = \lfloor 2^m \mu  \rfloor$, increase $u$ until $\bar{T}^+(u)\leq \Upsilon \bar{E}_{\rm proj}(m,\alpha)$.}
\STATE{Starting from $k = \lfloor 2^m \mu  \rfloor$, decrease $l$ until $\bar{T}^-(l)\leq \Upsilon \bar{E}_{\rm proj}(m,\alpha)$.}
\STATE{Let $\bar{E}_{\rm quad}(N_Q)$ denote \eqref{eq:quadD-bound} plus \eqref{eq:quadV-bound}. Increase the number of quadrature nodes $N_Q$ until $\bar{E}_{\rm quad}(N_Q)\le\Upsilon \bar{E}_{\rm proj}(m,\alpha)$.}
\STATE{\textbf{Outputs:} $\alpha,m,l,u,N_Q$ and $v_{m,\Lambda,Q}(x)$.}
\end{algorithmic}
\end{algorithm}

\section{Numerical experiments}\label{sec:numREsults}

In this section, we present numerical results for the proposed method and highlight the advantages of the damped SWIFT method over the non-damped version. In the absence of a closed-form solution, reference prices were computed by means of the COS method. Damped and undamped SWIFT are compared against the same reference value and target accuracy. We report the truncation range $l$ and $u$, the number of quadrature nodes \(N_Q\), and the resulting absolute error. We followed the strategy outlined in \cref{alg:dampedSWIFT} with
%$\epsilon=\hat{E}_{\rm proj}(m,\alpha)$ and 
$\Upsilon=0.1$. If the target accuracy is close to machine precision, we choose $l$ and $u$ such that the product of the coefficients is of the same order of accuracy. \ref{sup:boundsNumExam} illustrates with a numerical example the estimates provided in \cref{sec:Asymptotic decay of the payoff and density coefficients with respect to the translation parameters}.

In all the numerical experiments presented below, the damped method consistently requires fewer coefficients than the undamped formulation. This improvement is particularly pronounced in the region where the payoff is zero. For calls, this is the left tail; for puts, it is the right tail. This is consistent with the product-coefficient viewpoint: even if the density coefficients remain non-negligible in that region, the payoff coefficients significantly reduce \(|D^\alpha_{m,k}\cdot V^\alpha_{m,k}|\). Therefore, a truncation rule based on the product coefficients removes terms that density-only truncation rules tend to retain.

In \cref{tab:GBM}, we begin with the pricing of call and put options under the GBM model. Next, in \cref{tab:VG},  we present numerical results under the heavy-tailed VG model and examine the behavior as the resolution level $m$ increases.  Finally, in \cref{tab:VG_call_put}, we price long-maturity call and put options under the VG model. The large errors of the undamped SWIFT method for the \(T=100\) call are caused by the growth of the undamped call payoff coefficients on the right tail. Since \((e^y-K)^+\) grows exponentially as \(y\to\infty\), the corresponding payoff coefficients do not decay sufficiently fast over the wide truncation range required at long maturity. Increasing the number of retained terms can therefore amplify the unstable right-tail contribution rather than improve the approximation. In contrast, the damped method converges rapidly with few coefficients and remains stable when adding unnecessary terms. For put options, the non-damped method converges, though it requires more coefficients than the damped approach.

\begin{table}[!htbp]
\footnotesize
\centering
\caption{Pricing of call and put options under GBM with $S_0=K=100$, $r_0=0.1$, $T=1$, $\sigma=0.2$. 
%In both methods $m=4$ was considered.
%(the modulus of the characteristic function and of the characteristic function extended with damping at the corners of the domain are of the same order of magnitude). 
}
\label{tab:GBM}
\begin{tabular}{|r|r|r|r|r|r|r|r|}
\hline
$m=4$ & \multicolumn{3}{|c|}{SWIFT} & \multicolumn{4}{|c|}{Damped SWIFT, $N_Q = 120$} \\
\hline
Option & $l$ & $u$ & Abs error & $l$ & $u$ & $\alpha$ & Abs error \\
\hline
Call & $-30$ & $33$ & $1.9 \cdot 10^{-14}$ & $-17$ & $28$ & $6.45$ & $8.8 \cdot 10^{-15} $ \\
\hline
Put & $-31$ & $34$ & $3.5 \cdot 10^{-15}$ & $-26$ & $18$ & $-7.91$ & $2.7 \cdot 10^{-14} $ \\
\hline
\end{tabular}
\end{table}

%Next, in \cref{tab:CGMY}, we price call and put options under the CGMY model. In this case, the damped method, by truncating the series using also the information from the payoff coefficients, requires significantly fewer coefficients to achieve similar or even better accuracy than the non-damped method.

% \begin{table}[!htbp]
% \footnotesize
% \centering
% \caption{Pricing of call and put options under CGMY with $S_0=K=100$, $r=0.1$, $T=1$, $C=1$, $G=5$, $M=5$, $Y=0.1$. 
% %The scale $m=6$ was considered in both methods.
% %(the modulus of the characteristic function and of the characteristic function extended with damping at the corners of the domain are of the same order of magnitude). 
% %In the damped SWIFT version, $2^{11}$ quadrature nodes were considered. 
% %The reference prices are computed by means of the COS method with a very large number of cosine terms ($15.86966263787780$ for the call, and $6.353404530457659$ for the put).
% } \label{tab:CGMY}
% \begin{tabular}{|r|r|r|r|r|r|r|r|}
% \hline
% $m=6$ & \multicolumn{3}{|c|}{SWIFT} & \multicolumn{4}{|c|}{Damped SWIFT, $2^{11}$ quad points} \\
% \hline
% Option & $l$ & $u$ & Abs error & $l$ & $u$ & $\alpha$ & Abs error \\
% \hline
% Call & $-305$ & $312$ & $ 3.4\cdot 10^{-5}$ & $-10$ & $280$ & $3.56$ & $1.0 \cdot 10^{-7} $ \\
% \hline
% Put & $-305$ & $312$ & $ 9.7 \cdot 10^{-7}$ & $-220$ & $20$ & $-3.45$ & $7.5\cdot 10^{-7} $ \\
% \hline
% \end{tabular}
% \end{table}

\begin{table}[!htbp]
\footnotesize
\centering
\caption{Pricing of a call option under VG dynamics with parameters $S_0=K=100$, $r_0=0.02$, $T=1$, $\sigma=0.4$, $\theta=-0.3$, $\nu=0.5$. COS reference price $16.573188141372889$.
%In the damped SWIFT, the damping parameter is $\alpha=4.00$. 
%The reference price was computed by means of the COS method with a very large number of cosine terms ($16.573188141372889$).
}
\label{tab:VG}
\begin{tabular}{|r|r|r|r|r|r|r|r|r|}
\hline
 & \multicolumn{3}{|c|}{SWIFT} & \multicolumn{4}{|c|}{Damped SWIFT, $\alpha=4$} \\
\hline
$m$ & $l$ & $u$ & Abs error & $l$ & $u$ & $N_Q$ & Abs error \\
\hline
$3$ & $-35$ & $39$ & $ 6.4\cdot 10^{-2}$ & $-5$ & $22$ & $60$ & $3.0\cdot 10^{-4} $ \\
\hline
$4$ & $-69$ & $77$ & $1.9 \cdot 10^{-3}$ & $-10$ & $48$ & $120$ & $4.1 \cdot 10^{-6} $ \\
\hline
$5$ & $-138$ & $153$ & $4.7 \cdot 10^{-6}$ & $-20$ & $110$ & $270$ & $4.2 \cdot 10^{-7} $ \\
\hline
$6$ & $-276$ & $305$ & $ 5.8\cdot 10^{-7}$ & $-41$ & $238$ & $550$ & $1.6 \cdot 10^{-7} $ \\
\hline
$7$ & $-551$ & $609$ & $5.3 \cdot 10^{-8}$ & $-80$ & $518$ & $1190$ & $2.0 \cdot 10^{-8} $ \\
\hline
\end{tabular}
\end{table}

%Finally, we price long-maturity call and put options under VG in \cref{tab:VG_call_2}. In the case of maturity $T=100$, when pricing call options, the SWIFT method does not converge to the option price, as expected, due to the lack of decay in the payoff coefficients. In fact, even as the number of coefficients considered increases, the method diverges more and more. In contrast, the damped method converges to the option price with very few coefficients. Another advantage of the damped method is that the inclusion of more coefficients than strictly necessary, although useless, does not adversely affect its convergence. For put options, since the payoff is decreasing, the non-damped method does converge, but it requires more coefficients than the damped method.

\begin{table}[!htbp]
\footnotesize
\centering
\caption{Pricing of call and put options under VG dynamics $S_0=100$, $K=70$, $r_0=0.05$, $\sigma=0.3$, $\theta=-0.4$, $\nu=0.5$ for $T=5,100$. The reference COS prices are also shown.
%Reference prices were computed by means of the COS method with a very large number of cosine terms ($54.194215297153875$ for $T=5$, and $99.80412945414334$ for $T=100$). 
In the damped SWIFT, $\alpha_{\text{call}}=2.06$ and $\alpha_{\text{put}}=-1.21$ for $T=5$, and $\alpha_{\text{call}}=1.08$ and $\alpha_{\text{put}}=-0.20$ for $T=100$, $N_Q = 1200$. In the undamped SWIFT, $L$ refers to the original SWIFT truncation parameter.}
\label{tab:VG_call_put}
\begin{tabular}{|c|r|r|r|r||r|r|r|}
\hline
 \multicolumn{2}{|c|}{Call} & \multicolumn{3}{c||}{$T=5$, $m=4$} & \multicolumn{3}{c|}{$T=100$, $m=0$} \\
 \hline
 \multicolumn{2}{|c|}{COS} & \multicolumn{3}{c||}{$54.194215297153875$} & \multicolumn{3}{c|}{$99.80412945414334$} \\
 \hline
\multirow{4}{*}{\rotatebox{90}{SWIFT}} & $L$ & $l$ & $u$ & Abs error & $l$ & $u$ & Abs error\\
\cline{2-8}
%  & $8$ & $-83$ & $154$ & $3.08 \cdot 10^{-4}$ & $5$ & $72$ & $6.35 \cdot 10^{25}$\\
\cline{2-8}
& $10$ & $-112$ & $184$ & $6.5 \cdot 10^{-6}$ & $-4$ & $80$ & $2.2 \cdot 10^{29}$\\
\cline{2-8}
% & $12$ & $-142$ & $213$ & $4.65 \cdot 10^{-7}$ & $-12$ & $88$ & $3.28 \cdot 10^{31}$\\
\cline{2-8}
& $14$ & $-171$ & $243$ & $5.2 \cdot 10^{-9}$ & $-20$ & $96$ & $9.2 \cdot 10^{32}$\\
\cline{2-8}
 & $16$ & $-201$ & $272$ & $4.28 \cdot 10^{-8}$ & $-28$ & $105$ & $1.07 \cdot 10^{33}$\\
% \cline{2-8}
% & $18$ & $-230$ & $302$ & $1.4 \cdot 10^{-7}$ & $-37$ & $113$ & $4.7 \cdot 10^{32}$\\
\cline{2-8}
\hline 
 \multicolumn{2}{|c|}{Damped} & $l$ & $u$ & Abs error & $l$ & $u$ & Abs error\\
 \cline{3-8}
 \multicolumn{2}{|c|}{SWIFT} & $-82$ & $102$ & $6.3 \cdot 10^{-14}$ & $-18$ & $64$ & $2.8 \cdot 10^{-12}$\\
\hline
 \hline
 \multicolumn{2}{|c|}{Put} & \multicolumn{3}{c||}{$T=5$, $m=4$} & \multicolumn{3}{c|}{$T=100$, $m=0$} \\
 \hline
 \multicolumn{2}{|c|}{COS} & \multicolumn{3}{c||}{$8.710270112152223$} & \multicolumn{3}{c|}{$0.27578574407931616$} \\
 \hline 
\multirow{4}{*}{\rotatebox{90}{SWIFT}} & $L$ & $l$ & $u$ & Abs error & $l$ & $u$ & Abs error\\
\cline{2-8}
%  & $8$ & $-83$ & $154$ & $6.68 \cdot 10^{-4}$ & $5$ & $72$ & $1.23 \cdot 10^{0}$\\
\cline{2-8}
& $10$ & $-112$ & $184$ & $5.9 \cdot 10^{-6}$ & $-4$ & $80$ & $1.2 \cdot 10^{-1}$\\
\cline{2-8}
% & $12$ & $-142$ & $213$ & $3.35 \cdot 10^{-8}$ & $-12$ & $88$ & $3.30 \cdot 10^{-3}$\\
\cline{2-8}
& $14$ & $-171$ & $243$ & $1.8 \cdot 10^{-10}$ & $-20$ & $96$ & $7.5 \cdot 10^{-6}$\\
\cline{2-8}
% & $16$ & $-201$ & $272$ & $7.35 \cdot 10^{-13}$ & $-28$ & $105$ & $2.03 \cdot 10^{-9}$\\
\cline{2-8}
& $18$ & $-230$ & $302$ & $3.1 \cdot 10^{-14}$ & $-37$ & $113$ & $2.5 \cdot 10^{-14}$\\
\cline{2-8}
\hline 
 \multicolumn{2}{|c|}{Damped} & $l$ & $u$ & Abs error & $l$ & $u$ & Abs error\\
 \cline{3-8}
 \multicolumn{2}{|c|}{SWIFT} & $-237$ & $63$ & $7.1 \cdot 10^{-15}$ & $-41$ & $22$ & $1.0 \cdot 10^{-16}$\\
 \hline 
\end{tabular}
\end{table}

\begingroup
\appendix
\section{Proofs of \cref{lem:fourier box value}, \cref{lem:four box assumption with strip payoff} and \cref{lem:four box assumption with strip density}}
\label{app:lem:fourier box value}
\begin{proof}[Proof of \cref{lem:fourier box value}]
Repeated integration by parts gives
\[
	\frac{1}{2\pi}\int_{|u| \ge r}\hat{h}(u) e^{\ii uy}\dd u
	=
	\sum_{n=0}^{N}
	\frac{\ii ^n}{2\pi \ii \, y^{n+1}}
	\left[\hat{h}^{(n)}(-r) e^{-\ii r y}-\hat{h}^{(n)}(r) e^{\ii r y}\right]
	+\mathcal R_N(r, y),
\]
where the remainder is
\[
	\mathcal R_N(r, y)
	=
	\frac{\ii^{N+1}}{2\pi y^{N+1}}
	\int_{|u| \ge r}\hat{h}^{(N+1)}(u) e^{\ii u y}\dd u,
\]
and hence satisfies \eqref{eq:four box remainder-bound}. Since $h$ is real-valued, it follows that $\overline{\ii^n \hat{h}^{(n)}(r)}=\ii^n \hat{h}^{(n)}(-r).$
We now substitute this symmetry into the boundary terms. If $z= \ii^n \hat{h}^{(n)}(r) e^{ir y}$ then
\begin{align*}
	\frac{\ii^n}{2\pi \ii  y^{n+1}}
	\left[\hat{h}^{(n)}(-r) e^{-\ii r y}-\hat{h}^{(n)}(r) e^{\ii r y}\right]
	&=\frac{1}{\pi y}\left(\frac{\bar{z}-z}{2\ii}\right) y^{-n}\\
    &=-\frac{1}{\pi y}\Im[z]\,  y^{-n} =
	-\frac{1}{\pi y}\Im\!\left[ \ii^{n} e^{\ii r y}\hat{h}^{(n)}(r) \right] y^{-n}.
\end{align*}
Summing over $n=0,\ldots,N$ proves \cref{lem:fourier box value}.
\end{proof}
\begin{proof}[Proof of \cref{lem:four box assumption with strip payoff}]
    Since $\alpha$ is in the interior of $\delta_h$, we can find a small $ \epsilon >0 $ such that $ e^{\epsilon|u|}h_\alpha(u) \in  L^{1}(\R) $.  Consequently, $ yh_\alpha(y)\in  L^{1}(\R) $ and $\hat{h}_\alpha \in C^1(\R)$. By the Lebesgue-Riemann lemma, we obtain that 
$$ \lim_{|u|\to \infty}\left|\hat{h}_\alpha^{(1)}(u)\right| = \lim_{|u|\to \infty}\left|\widehat{y h_\alpha (y)}(u)\right| =0,$$
and this shows that the first assumption of \cref{lem:fourier box value} is satisfied. Since $e^{\epsilon|u|}h_\alpha(u) \in  L^{1}(\R)$ we also have that $y^n h_\alpha(y)\in L^1(\R)$ for all $n\in \mathbb{N}$, and $\widehat{y h_{\alpha} (y)}(u):=\hat{g}(u)$ is smooth with derivatives vanishing at infinity. By the assumption $\alpha \in \delta_{h}\cap \delta_{h^{(1)}}\cap  \delta_{h^{(2)}}$, it follows that $g, g^{(1)}$ and $ g^{(2)}$ belong to $L^1(\R)$. Using integration by parts, we obtain that
$\hat{h}_\alpha^{(1)}(u) = - \ii \widehat{g}(u)  = \ii u^{-2} \widehat{g^{(2)}}(u).$
Since $g^{(2)}\in L^1(\R)$, by the Lebesgue-Riemann lemma $\widehat{g^{(2)}}(u)$ is bounded and this proves that $\hat{h}_\alpha^{(1)}\in L^1(|u|>r)$ for any $r>0$. Accordingly, the second assumption of \cref{lem:fourier box value} is also satisfied.
\end{proof}
\begin{proof}[Proof of \cref{lem:four box assumption with strip density}]
Since $\alpha$ is in the interior of $\delta_h$, we can find a small $ \epsilon >0 $ such that $ e^{\epsilon|u|}h_\alpha(u) \in  L^{1}(\R) $. It follows that $\hat{h}_\alpha$ can be extended to an analytic function in the strip $\{u\in \C:\, |\Im(u)|<\epsilon\}$. For every $u\in \R$ and $\eta <\epsilon$, by the Cauchy's integral formula, we get that 
\begin{equation*}
    \begin{aligned}
        \hat{h}_\alpha^{(n)}(u) = \frac{n!}{2\pi \ii}\int_{|w-u|=\eta} \frac{\hat{h}_\alpha(w)}{(w-u)^{n+1}} \dd w =\frac{n!}{2\pi \eta^{n}}\int_0^{2\pi} \hat{h}_\alpha ( u + \eta e^{\ii \theta}) e^{-\ii \theta n} \dd \theta.
    \end{aligned}
\end{equation*}
Accordingly, by Tonelli's theorem
\begin{equation*}
    \begin{aligned}
       \int_{-\infty}^{\infty} |\hat{h}_\alpha^{(n)}(u)|\dd u &\leq  \frac{n!}{2\pi \eta^{n}}\int_{-\infty}^{\infty}\int_0^{2\pi} |\hat{h}_\alpha ( u + \eta e^{\ii \theta})| \dd \theta \dd u\\
       & = \frac{n!}{2\pi \eta^{n}}\int_0^{2\pi} \int_{-\infty}^{\infty}|\hat{h}_\alpha ( u + \eta e^{\ii \theta})| \dd u \dd \theta \\
       & = \frac{n!}{2\pi \eta^{n}}\int_0^{2\pi} \int_{-\infty}^{\infty}|\hat{h}( u + \ii \alpha  +\eta e^{\ii \theta})| \dd u \dd \theta\\
       & \leq \frac{n!}{\eta^{n}}  \sup_{|w-\alpha|<\epsilon} \|\hat{h}(\cdot + \ii w)\|_1 <\infty.
    \end{aligned}
\end{equation*}
This shows that $\hat{h}_\alpha$ is smooth on the real line, that all of its derivatives are in $L^1(\R)$, and hence they vanish at infinity. Therefore, all the assumptions of \cref{lem:fourier box value}
are satisfied.
\end{proof}
\section{Proof of \cref{teo:estimate fat tails}}
\label{app:teo:estimate fat tails}
\begin{proof}
  %Since $0\in \delta_h$, then $h\in L^1(\R)$, and 
  As $h$ is piecewise $C^1$ and in $L^1(\R)$ then $I_{\infty}[\hat{h}](y)$ is well-defined (see \cite[Theorem 2.14]{jerri1992integral}).
Consider now the extension $ \hat{h} = Z + H$ defined on the simply connected open set $ \C \setminus \{  \gamma_{u_1}, \allowbreak \ldots,\allowbreak \gamma_{u_N} \} $, where $ u_1, \ldots  , u_N $ are the branch points of $Z$ of orders  $ \rho_1, \ldots , \rho_N $. We assume that $u_1, \ldots, u_N$ are algebraic and give insights on how to deal with logarithmic singularities when we deem it appropriate. Notice that, since $ h \in  L^{1}(\R) $ we have that $ \hat{h}(u) $ is defined for every $ u \in \R $, and hence there are no real branch points for $ Z $.
%BAD ALTERNATIVE
%\textcolor{red}{Since $\delta_h =(-a,b)$, then $\hat{h}$ is analytical in the strip $\{u\in \C:\, -a<\Im(u)<b\}$. In particular, if $\Im(u_n)>0$ then $\Im(u_n)\geq b$, and if $\Im(u_n)<0$ then $\Im(u_n)\leq -a$. Moreover,  we have that $\hat{h}$ has controlled growth and $\hat{h}(R+i\alpha)\to 0$ for $R\to \infty$ and $\alpha \in (-a,b)$. By the Phragmén-Lindelöf principle, then $\hat{h}$ uniformly to zero at infinity on the strip.}

We now aim to estimate $I_{\infty}[\hat{h}](y)$ via ``The Method of Steepest Descents'' (see \cite[Pages 137-140]{miller2006applied}).

%\footnote{How could it be done? If $ \Im(z_n)>0 $ then the branch cut of $ (u-u_n)^{-\rho_n} $ starts at $ z_n $ and is above the real line, as the one given in \eqref{eq:proof est fat tails:branch cut shape}. Define now $ (u-u_n)^{-\rho_n} $ on the domain given by \eqref{eq:proof est fat tails:branch cut shape}, which is always possible. On the real line, this latter function differs from the ``original'' $ (u-u_n)^{-\rho_n} $ by a multiplicative constant $ e^{-2\pi \ii k \rho_n} $, for some $ k \in \Z $. Accordingly, by multiplying the ``new" $ (u-u_n)^{-\rho_n} $ with a constant, we found again the ``original" function on the real line.}

From now on, we assume that $y>0 $ as the proof for $ y < 0 $ is analogous. Therefore, for the case $y>0$, we work only with the branch points $u_n$ with positive imaginary part. To estimate
$I_{\infty}[\hat{h}](y) $
 we consider the red path $ \gamma_R $, whose shape is shown in \cref{fig:proof fat tail estimation 2}. This path is divided into three main parts:
\begin{itemize}
  \item[\it a)] The real interval $ [-R,R] $.
  \item[\it b)] The segments
  \begin{align*}
  \{ R+ e^{\ii \theta } u :\, 0\leq u \leq \csc(\theta) R\}&\cup \{ u + \ii R :\, u\in \R, |u|\leq  R(1+\cot(\theta))\} \\
  & \cup \{ -R - e^{-\ii \theta } u :\, 0\leq u \leq \csc(\theta) R\},
  \end{align*}
where $ \theta \in  (0,\frac{\pi}{2}) $ is chosen according to assumption $ ii) $. We assume that $ R $ is sufficiently big so that no branch cut is crossed.
  \item[\it c)] For each $ n=1, \ldots , N$ such that $\Im(u_n)>0$, the U-shaped path $ U_n $ surrounding $ u_n $ at distance $ \epsilon= R^{-1}$ from  $ \gamma_{u_n} $.
\end{itemize}

Since $ \hat{h} $ is holomorphic inside $ \gamma_R $, we have $ \int_{\gamma_R} \hat{h}(u)e^{\ii uy}\dd u=0$,
and
$$I_{\infty}[\hat{h}](y)= - \frac{1}{2\pi}\lim_{R \to \infty}\int_{\gamma_R\setminus[-R,R]} \hat{h}(u)e^{\ii uy}\dd u.  $$
\fbox{ \it Estimation along part $ b) $:} \\

The integral along the pieces specified by $ b) $ is given by
\begin{align}
      e^{\ii\theta}\int_0^{R \csc(\theta)} \hat{h}(R + e^{\ii\theta} u)e^{\ii y(r+e^{\ii\theta} u)} \dd u &- \int_{-R(1+\cot(\theta))}^{R(1+\cot(\theta))}  \hat{h}(u+ \ii R)e^{\ii y(u+\ii R)} \dd u\nonumber\\
      &- e^{-\ii\theta}\int_{0}^{R\csc(\theta)} \hat{h}(-R - e^{-\ii\theta}u)e^{-\ii y(R +e^{-\ii\theta}u)} \dd u.\label{eq:proof fat tail: int part b)}
\end{align}
By the polynomial growth assumption, we have
$  |\hat{h}(u)| \leq C (1+|u|^{q}), $
and
\begin{align*}
    \left| \int_{-R(1+\cot(\theta))}^{R(1+\cot(\theta))}  \hat{h}(u+ \ii R)e^{\ii y(u+\ii R)} \dd u \right| &\leq C \int_{-R(1+\cot(\theta))}^{R(1+\cot(\theta))} (1+|u|^q + R^q) e^{-yR} \dd u \\
    &= 2C e^{-yR} \int_{0}^{R(1+\cot(\theta))} (1+|u|^q + R^q) \dd u,
\end{align*}
which tends to zero as the integral term has at most polynomial growth wrt $ R $. By assumption $ ii) $, the integrals along the lateral edges can be bounded uniformly wrt $R$. Since the supremum of $ \hat{h} $ on the sloped half-lines goes to zero, then by the dominated convergence theorem, these integrals tend to zero.\\

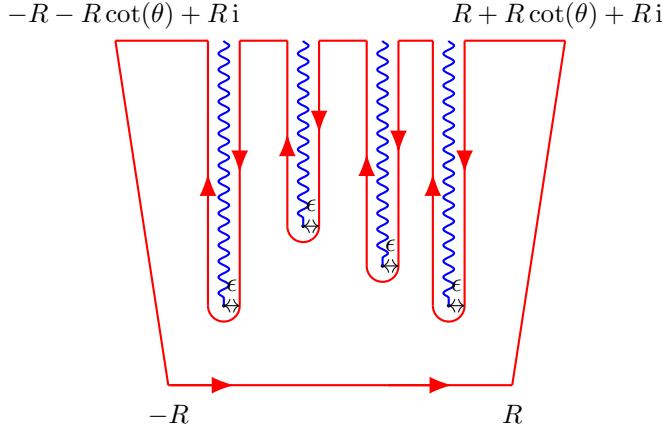
\begin{figure}[!ht]
\centering
\begin{tikzpicture}[
  scale=0.35,
  wall/.style={red, thick},
  flow/.style={red, thick, -{Latex[length=3mm]}},
  waveline/.style={blue, thick, decorate,
    decoration={snake, amplitude=2pt, segment length=6pt}}
]

% --- parámetros ---
\def\W{13}
\def\H{\W}
\def\r{0.6}
\def\s{2}   % apertura superior del trapecio

% posiciones de las U
\def\xA{1.5}
\def\xB{4.5}
\def\xC{7.5}
\def\xD{10}

% --- base ---
\draw[wall] (0,0)--(\W,0);

% --- laterales del trapecio ---
\draw[wall] (0,0)--(-\s,\H);
\draw[wall] (\W,0)--(\W+\s,\H);

% --- techo con huecos ---
\draw[wall] (-\s,\H)--(\xA,\H);
\draw[wall] (\xA+2*\r,\H)--(\xB,\H);
\draw[wall] (\xB+2*\r,\H)--(\xC,\H);
\draw[wall] (\xC+2*\r,\H)--(\xD,\H);
\draw[wall] (\xD+2*\r,\H)--(\W+\s,\H);

% --- etiquetas ---
\node[below=4pt] at (0,0) {$-R$};
\node[below=4pt] at (\W,0) {$R$};

\node[above=1pt, xshift=3pt]  at (-\s,\H) {$-R -R \cot(\theta) + R\,\mathrm{i}$};
\node[above=1pt, xshift=-3pt] at (\W+\s,\H) {$R + R \cot(\theta)+ R\,\mathrm{i}$};

% flechas inferiores
\foreach \k in {0.3,2.2}{
  \pgfmathsetmacro{\x}{\k*\W/3}
  \draw[flow] (\x-1.2,0)--(\x+1.2,0);
}

% --- macro U ---
\newcommand{\Ushape}[2]{%
  % paredes de la U
  \draw[wall] (#1,\H)--(#1,#2);
  \draw[wall] (#1+2*\r,\H)--(#1+2*\r,#2);

  % arco inferior
  \draw[wall] (#1,#2) arc(180:360:\r);

  % punto negro
  \pgfmathsetmacro{\px}{#1+\r}
  \pgfmathsetmacro{\py}{#2}
  \fill (\px,\py) circle(2pt);

  % línea ondulada
  \draw[waveline] (\px,\H)--(\px,\py);

  % flechas verticales
  \pgfmathsetmacro{\mid}{(#2+\H)/2}
  \draw[flow] (#1,\mid-0.8)--(#1,\mid);
  \draw[flow] (#1+2*\r,\mid+0.8)--(#1+2*\r,\mid);

  % flecha epsilon — corregida
  \draw[<->] (\px,\py) -- (\px+\r,\py)
     node[midway, above=2pt] {$\epsilon$};
}

% --- dibujar las U ---
\Ushape{\xA}{3}
\Ushape{\xB}{6}
\Ushape{\xC}{4.5}
\Ushape{\xD}{3}

\end{tikzpicture}
\caption{Integration path in the upper half-plane. The black points are the branch points contained in $ \{ \Im(u)>0 \} $. The waved-blue lines represent branch cuts of $ \hat{h} $.}
\label{fig:proof fat tail estimation 2}
\end{figure}

\fbox{\it Part $ c) $: Branch points}\\

Now we aim to show that for all $U_n$ with $\Im(u_n)>0$, we can find $ C>0 $ such that
\begin{equation*}
%\label{eq:proof est fat tails: contr alg sing to prove}
    \left| \lim_{R\to \infty}\int_{U_n} \hat{h}(u)e^{\ii uy} \dd u\right| \leq  \frac{C}{\Gamma(\rho_n)} y^{\Re(\rho_n)-1} e^{-\Im(u_n)y}
\end{equation*}
when $ y>1 $. We outline that this estimate holds for all $ \rho_n \in \C \setminus \{ 0 \}$ when $ u_n $ is an algebraic branch point. In the logarithmic case, it remains valid with $\rho_n$ being a nonpositive integer.

Firstly, let $ p_n = \max(0,\lfloor \Re(\rho_n) \rfloor) $. Notice that if $ u \in  U_n $ then $\hat{h} $ is analytic in $ B(u,\epsilon=R^{-1}) $ and by the Cauchy's estimates we find that
$$ |\hat{h}^{(k)}(u)| \leq k! R^{k} \sup_{|s|\leq R^{-1}}|\hat{h}(u+s)|.$$
By the polynomial growth for $ \hat{h} $ we have
$$ |\hat{h}^{(k)}(u)| \leq  C k! R^{k} (|u|^{q}+ R^{-q}) \leq C k! R^{k} (|u|^{q}+ 1),$$
and a similar bound along $ U_n $ holds for
$ \widetilde{g}(u):=\hat{h}(z)(u-u_n)^{\rho_n}$, and its derivatives. In particular, if $ u \in  \partial U_n = \{ u_n+\epsilon(R) + \ii R, u_n-\epsilon(R) + \ii R \}$ then for every $ k \in  \mathbb{N} $
\begin{equation}
  \label{eq:proof fat tail est: decay g tilde boundary}
  |\widetilde{g}^{(k)}(u)e^{\ii uy}|\leq C k! R^{k}(R^{q}+1) e^{-yR} \xrightarrow{R\to \infty}0 , \qquad  u \in  \partial U_n.
\end{equation}

\fbox{\it Integration by parts:}\\

If $ p_n >0$, along $ U_n $ we have
  $$  \int_{U_n} \hat{h}(u)e^{\ii uy}\dd u = \int_{U_n} \widetilde{g}(u){(u-u_n)^{-\rho_n}}e^{\ii uy}\dd u.$$
Integrating by parts, we obtain
\begin{align*}
  &\int_{U_n} \widetilde{g}(u){(u-u_n)^{-\rho_n}}e^{\ii uy}\dd u=\frac{1}{\rho_n-1}\Bigg(-\int_{\partial U_n} {(u-u_n)^{-\rho_n+1}}\widetilde{g}(u)e^{\ii uy} \dd\sigma(u) \\
  &\hspace{6cm} + \int_{U_n} (u-u_n)^{-\rho_n+1} \frac{\mathrm{d}}{\mathrm{d}u}(\widetilde{g}(u)e^{\ii uy}) \dd u\Bigg).
    %\\
    %&= \sum_{0\leq k<p_n}(-1)^{k}  \int_{\partial U_n} (u-u_n)^{k+1}\frac{\mathrm{d}^k}{\mathrm{d}z^k}\left(\hat{f}(u)e^{iux}\right) d\sigma(u) \\
    %&+ (-1)^{p_n}\int_{U_n} (u-u_n)^{p_n} \frac{\mathrm{d}^{p_n}}{\mathrm{d}z^{p_n}}\left(\hat{f}(u)e^{iux}\right) du
\end{align*}
By \eqref{eq:proof fat tail est: decay g tilde boundary}, the integrand vanishes exponentially on the boundary and
$$ \lim_{R\to \infty} \int_{U_n} \hat{h}(u)e^{\ii uy}\dd u =\frac{1}{\rho_n-1} \lim_{R\to \infty}\int_{U_n} (u-u_n)^{-\rho_n+1} \frac{\mathrm{d}}{\mathrm{d}u}(\widetilde{g}(u)e^{\ii uy}) \dd u .$$
We can repeat the integration by parts $ p_n$ times to obtain
\begin{equation*}
    \lim_{R\to \infty} \int_{U_n} \hat{h}(u)e^{\ii uy}\dd u =\frac{\Gamma(\rho_n-p_n)}{\Gamma(\rho_n)}\lim_{R\to \infty}\int_{U_n} (u-u_n)^{-\rho_n+p_n} \frac{\mathrm{d}^{p_n}}{\mathrm{d}u^{p_n}}(\widetilde{g}(u)e^{\ii uy}) \dd u.
\end{equation*}
Moreover,
$$ \frac{\mathrm{d}^{p_n}}{\mathrm{d}u^{p_n}}(\widetilde{g}(u)e^{\ii uy})=\left(\sum_{j=0}^{p_n} \binom{p_n}{j} \widetilde{g}^{(p_n-j)}(u)(\ii y)^{j}\right) e^{\ii uy},$$
and hence
\begin{align}
    &\lim_{R\to \infty} \int_{U_n} \hat{h}(u)e^{\ii uy}\dd u = \nonumber\\
&\hspace{1cm}\frac{\Gamma(\rho_n-p_n)}{\Gamma(\rho_n)}\sum_{j=0}^{p_n} \binom{p_n}{j} (\ii y)^{j}  \left(\lim_{R\to \infty}\int_{U_n} (u-u_n)^{-\rho_n+p_n} \widetilde{g}^{(p_n-j)}(u)e^{\ii uy} \dd u\right).\label{eq:proof est fat tails:overall contribution}
\end{align}
Our next step will be the estimation of each
\begin{equation}
  \label{eq:proof est fat tails:algebraic sing contribution}
\lim_{R\to \infty}\int_{U_n} (u-u_n)^{-\rho_n+p_n} \widetilde{g}^{(p_n-j)}(u)e^{\ii uy} \dd u.
\end{equation}
To this aim, consider the function $ \widetilde{g}(u)= \hat{h}(u)(u-u_n)^{\rho_n} $. By assumption, we have that there is an entire function $ H $ such that $ \hat{h}-H = Z$ has a branch point of order $ \rho_n $ at $ u_n $. As a consequence, we can find $ \delta_n>0 $ such that
$$ \widetilde{g}(u)=\sum_{m=0}^{\infty} a_m (u-u_n)^{m} + \sum_{m=0}^{\infty} b_m (u-u_n)^{m+\rho_n} , \qquad  u \in  B(u_n,\delta_n).$$
Deriving these series, we find that for $ 0\leq j \leq  p_n $, the derivative $  \widetilde{g}^{(j)}(u) $ has the form
$$  \widetilde{g}^{(j)}(u)=\sum_{m=0}^{\infty} \tilde{a}_{m+j} (u-u_n)^{m} + \sum_{m=0}^{\infty} \tilde{b}_m (u-u_n)^{m+\rho_n- j} , \qquad  u \in  B(u_n,\delta_n), $$
and
\begin{align*}
(u-u_n)^{-\rho_n+p_n}\widetilde{g}^{(j)}(u)=&(u-u_n)^{-\rho_n+p_n} \sum_{m=0}^{\infty} \tilde{a}_{m+j} (u-u_n)^{m} \\
& + \sum_{m=0}^{\infty} \tilde{b}_m (u-u_n)^{m- j+ p_n} := (u-u_n)^{-\rho_n+p_n} g(u) + t(u),
\end{align*}
where $ g $ and $ t $ are both analytic in $ B(u_n,\delta_n) $.
Namely, we estimate \eqref{eq:proof est fat tails:algebraic sing contribution} assuming that it has the form
$$\lim_{R\to \infty}\int_{U_n} ((u-u_n)^{-\rho_n}g(u) + t(u))e^{\ii uy} \dd u $$
with $ \Re(\rho_n)< 1 $, the integrand having polynomial growth and $g,t$ being analytic in a neighborhood containing $ U_n $ and $ u_n $. Notice that in the logarithmic case, no integration by parts is required and $(u-u_n)^{-\rho_n}g(u)$ is replaced by $(u-u_n)^{-\rho_n} \log(u-u_n)g(u)$, with $\rho_n$ being a nonpositive integer.
\\

\fbox{\it Case: $ \Re(\rho_n)<1 $}\\

The distance between $ U_n $ and $ \gamma_{u_n}$ is given by $ \epsilon=R^{-1} $. Since
$$ U_n= [u_n + \epsilon + \ii R, u_n + \epsilon] \cup \{ u_n + \epsilon e^{-\ii\theta}:\, \theta \in  [0,\pi] \} \cup [u_n - \epsilon, u_n - \epsilon+ \ii R],$$
we have
\begin{align*}
    & \int_{U_n} \left((u-u_n)^{-\rho_n}g(u)+ t(u)\right)e^{\ii uy} \dd u = \ii e^{\ii y u_n} \Bigg{(} \\
    & \hspace{1.4cm }e^{\ii y\epsilon}\int_R^{0} \left((\epsilon + \ii u)^{-\rho_n}g(u_n + \epsilon+ \ii u)+ t(u_n + \epsilon+ \ii u)\right)e^{-uy}\dd u \\
    %& \hspace{1cm } -\epsilon^{1-\rho_n} \int_0^{\pi} \left( e^{\ii\theta \rho_n} g(u_n + \epsilon e^{-\ii\theta})+ t(u_n + \epsilon e^{-\ii\theta})\right)e^{\ii y \epsilon(\cos \theta - \ii \sin \theta)}\dd u \\
    & \hspace{1cm } - \epsilon \int_0^{\pi} \left( \epsilon^{-\rho_n} e^{\ii\theta \rho_n} g(u_n + \epsilon e^{-\ii\theta})+ t(u_n + \epsilon e^{-\ii\theta})\right)e^{-\ii\theta + \ii y \epsilon(\cos \theta - \ii \sin \theta)}\dd \theta \\    
    & \hspace{1cm}+ e^{-\ii y\epsilon} \int_0^{R}\left((\ii u-\epsilon)^{-\rho_n}g(u_n - \epsilon+ \ii u) + t(u_n - \epsilon+ \ii u)\right) e^{-uy} \dd u \Bigg{)}.
\end{align*}
As $ \Re(\rho_n)<1 $, the integral $\epsilon \int_0^\pi$ tends to zero as $\epsilon\to 0$.
%%\footnote{Notice that we are neglecting when $ \alpha_n $ is purely imaginary. In this case the term $ x^{\alpha-1} $ is just substituted by a constant. I skipped this detail.} then $ \epsilon^{1-\alpha_n}\to 0 $ as $ \epsilon\to 0$, while the argument in the second integral is bounded. Accordingly
%$$  \epsilon^{1-\rho_n} \int_0^{\pi} \left( e^{\ii\theta \rho_n} g(u_n + \epsilon e^{-\ii\theta})+ t(u_n + \epsilon e^{-\ii\theta})\right)e^{\ii y \epsilon(\cos \theta - \ii \sin \theta)}\dd u \xrightarrow{\epsilon\to 0}0.$$
%$$- \epsilon \int_0^{\pi} \left( \epsilon^{-\rho_n} e^{\ii\theta \rho_n} g(u_n + \epsilon e^{-\ii\theta})+ t(u_n + \epsilon e^{-\ii\theta})\right)e^{-\ii\theta + \ii y \epsilon(\cos \theta - \ii \sin \theta)}\dd \theta
% \xrightarrow{\epsilon\to 0}0.$$
%NEW VERSION STARTS HERE 
To estimate the integral $\int_0^R$, we split the interval as $\int_0^\delta + \int_\delta^R$, for $\delta >0$. Since the integrand has at most polynomial growth, we have
\begin{equation*}
    \begin{aligned}
      \Bigg{|}\int_\delta^{R}\Big{(}( \ii u&- \epsilon)^{-\rho_n}g(u_n - \epsilon + \ii u) + t(u_n- \epsilon + \ii u)\Big{)}e^{-uy}\dd u \Bigg{|}\leq\\
      &\leq 
       \int_\delta^{\infty}\left|( \ii u- \epsilon)^{-\rho_n}g(u_n - \epsilon + \ii u)+ t(u_n- \epsilon + \ii u)\right|e^{-uy}\dd u  \\
      &\leq C \int_\delta^\infty (1+|\epsilon|^q + |u_n|^q + |u|^q)e^{-u y}\dd u \leq C \int_\delta^\infty (1 + |u|^q)e^{-u y}\dd u\\
      &= C e^{-\delta y} \int_0^\infty (1+|w+ \delta|^q ) e^{-wy}\dd w  \leq C e^{-\delta y} \int_0^\infty (1+|w|^q+ |\delta|^q ) e^{-wy}\dd w \\
      &\leq C y^{-1} e^{-\delta y} (1+ y^{-q} + \delta^q),
    \end{aligned}
\end{equation*}
and as $\delta>0$ can be arbitrarily big, this term is $o(y^{-n})$ for every $n\geq 0$. The same holds for the integral $\int_R^0$ where the function $(u-u_n)^{-\rho_n} g(u) + t(u)$ is evaluated on the other side of the branch cut. 
Consequently, it remains to study the integral $\int_0^\delta$ with the integrand being evaluated on both sides of the branch cut. 
For $R \to \infty$ we have that $\epsilon\to 0 $, and $\ii e^{\ii y u_n}(\int_{\delta}^0 + \int_{0}^\delta)$ tends to
\begin{equation*}
  \ii e^{\ii y u_n} \left( \int_0^{\delta} \left[(\ii u)_-^{-\rho_n}- (\ii u)_+^{-\rho_n}\right] g(u_n + \ii u)e^{-uy} \dd u  \right).
\end{equation*}
Notice that, by continuity, we have  $ t(u_n+\ii u_-)-t(u_n+\ii u_+)=0 $.
The terms in the difference $ (\ii u)_-^{-\rho_n}- (\ii u)_+^{-\rho_n} $ are taken from two sides of the branch cut (see \cite[Sec. 4.8]{miller2006applied}). Therefore, there is some $ k \in \Z $ such that
\begin{align*}
  \ii e^{\ii y u_n} &\left( \int_0^{\delta} \left[(\ii u)_-^{-\rho_n}- (\ii u)_+^{-\rho_n}\right] g(u_n + \ii u)e^{-uy} \dd u  \right) \\
    &= \ii e^{\ii y u_n} e^{2\pi \ii k \rho_n} \left(e^{-\ii\pi\rho_n}- e^{\ii \pi \rho_n}\right)\int_0^{\delta} u^{-\rho_n}g(u_n + \ii u)e^{-uy}\dd u\\
    &= e^{\ii y u_n} e^{2\pi \ii k \rho_n}2 \sin(\pi \rho_n) \int_0^{\delta} u^{-\rho_n}g(u_n + \ii u)e^{-uy}\dd u.
\end{align*}
This last integral can be estimated as in the proof of Watson's Lemma (see \cite[Eq. (2.11)-2.12]{miller2006applied}) that gives us that for $ y>1 $
\begin{equation}
\label{eq:proof fat tail est: contr watson lemma}
\int_0^{\delta} u^{-\rho_n}g(u_n + \ii u)e^{-uy}\dd u = g(u_n)\Gamma(1-\rho_n)y^{\rho_n-1} + o(y^{\rho_n-1}),
\end{equation}
where
$|o(y^{\rho_n-1})|\leq C y^{\Re(\rho_n)-2}$, $y>1$.
As a result, we can find $ C>0 $ such that, for $ y>1 $,
\begin{align}
    \left| \lim_{R\to \infty}\int_{U_n} (u-u_n)^{-\rho_n}g(u)e^{\ii uy} \dd u\right|&\leq C e^{-\Im(u_n)y} |\sin(\pi \rho_n)\Gamma(1-\rho_n)| y^{\Re(\rho_n)-1}\nonumber\\
    &= C e^{-\Im(u_n)y} y^{\Re(\rho_n)-1} \frac{\pi}{|\Gamma(\rho_n)|}.\label{eq:proof fat tail est: contr alg sin gamma function}
\end{align}
For logarithmic branch points, as $\rho_n$ is an integer, the difference between the two sides of the branch cut of $(\ii u)^{-\rho_n} \log(\ii u)$ is due only to the logarithm. Therefore, this difference is equal to $2\pi \ii (\ii)^{-\rho_n}|u|^{-\rho_n}  $ and can be studied as in the algebraic case.\\

\fbox{\it Overall contribution:}\\

By combining \eqref{eq:proof est fat tails:overall contribution} with the last estimation \eqref{eq:proof fat tail est: contr alg sin gamma function}, we find that
\begin{align*}
    \left| \lim_{R\to \infty}\int_{U_n} \hat{h}(u)e^{\ii uy} \dd u\right| &\leq C \frac{ e^{-\Im(u_n)y}}{|\Gamma(\rho_n)|} y^{\Re(\rho_n)-p_n-1} \underbrace{\sum_{j=0}^{p_n}\binom{p_n}{j} y^{j}}_{=(1+y)^{p_n}} \\
    &\leq \frac{C}{|\Gamma(\rho_n)|} y^{\Re(\rho_n)-1} e^{-\Im(u_n)y}, \qquad  y>1.
\end{align*}
By summing these estimates over all the singularities of $ \hat{h} $ in the upper half-plane, we obtain that
\begin{equation*}
  \left|\frac{1}{2\pi}\int_{-\infty}^{\infty}\hat{h}(u)e^{\ii uy}\dd u\right|\leq  C \left( \sum_{\Im(u_n)>0} y^{\Re(\rho_n)-1} e ^{-\Im(u_n)y}\right), \qquad  y>1,
\end{equation*}
hence, the statement in \eqref{eq:est fourier inversion}  follows by assumption $iii)$ of this theorem. %expressed in \eqref{eq:assumption singularities position}.
%The estimation for $|f(x)|$ in Equation \eqref{eq:estimate fat tails} is finally obtained by letting $r\to \infty$ \textbf{and requires assumptions that make the Fourier inversion formula true} (e.g. $f$ continuous and with piecewise continuous derivatives.)
\end{proof}

\begin{remark}
\cref{teo:estimate fat tails} still holds when the extension of $\hat{h}$ presents two branch points lying on the same branch cut. In this case, the contribution contained in \eqref{eq:proof fat tail est: contr watson lemma} has to be considered for both branch points.
\end{remark}

\endgroup

\section*{Acknowledgments}
J. G. L\'opez-Salas acknowledges the support received from CITIC for a research stay at Utrecht University. CITIC, as a center accredited for excellence within the Galician University System and a member of the CIGUS Network, receives subsidies from the Department of Education, Science, Universities, and Vocational Training of the Xunta de Galicia. Additionally, it is co-financed by the EU through the FEDER Galicia 2021-27 operational program (Ref. ED431G 2023/01).

\bibliographystyle{siamplain}
\bibliography{references}

\newpage

\thispagestyle{empty}
\setcounter{page}{1}
\setcounter{figure}{0}
\renewcommand{\thefigure}{SM\arabic{figure}}
\renewcommand{\figurename}{Fig.}

\headers{Trevisani, L\'opez-Salas, Ben Hammouda and Oosterlee}{SUPPLEMENTARY MATERIAL: DAMPED SWIFT }

\begin{center}
\textbf{SUPPLEMENTARY MATERIAL: A DAMPED SWIFT METHOD FOR EUROPEAN OPTION PRICING: COEFFICIENTS DECAY, TRUNCATION, AND ERROR ANALYSIS}
\end{center}

\vspace{1em}

\setcounter{section}{0}
\renewcommand{\thesection}{SM\arabic{section}}

\section{Alternative assumptions for $f_\alpha$ and $P_\alpha$}\label{sm:proof Parseval shannon}
As mentioned in the article, the $L^2$ condition in Assumption \ref{ass:noempty strip intersection} is required
only for simplicity. To pass from the physical space to the Fourier representation of the Fourier coefficients, one can directly apply the following proposition to $f_\alpha$ and $P_\alpha$. The proof below takes advantage of the Fourier inversion formula for piecewise $C^1$ functions, and the fact that the FT of the Shannon basis function has compact support.
 \begin{proposition}
  \label{prop: Parseval identity shannon wavelets}
  Let $ h \in L^{1}(\R) $ such that $ h $ and its first derivative are piecewise continuous. Then for $ m,k \in  \Z $, we have
  \begin{equation*}
    %\label{eq:parseval shannon wavelet}
    \int_{\R} h(y)\phi_{m,k}(y)\, \dd y= \frac{1}{2\pi} \int_{\R} \hat{h}(u) \overline{\hat{\phi}_{m,k}(u)}\, \dd u.
  \end{equation*}
\end{proposition}

 \begin{proof}
    As $ h \in L^{1}(\R) $ and $ \phi_{m,k} \in  L^{\infty}(\R) $, we have $ \int_{\R} |h(y)\phi_{m,k}(y)|\, \dd y< \infty$. Since $ h $ and its derivative are piecewise continuous, by \cite[Theorem 2.14]{jerri1992integralSM}, the Fourier inversion formula holds almost everywhere for $ h $. As a result, we have
\begin{align*}
    \int_{\R} h(y)\phi_{m,k}(y)\, \dd y &= \frac{1}{2\pi} \int_{\R} \left(\int_{\R} \hat{h}(u) e^{\ii u y}\, du\right) \phi_{m,k}(y)\, \dd y\\
    & = \frac{1}{2\pi} \lim_{R\to \infty}\int_{\R} \left(\int_{-R}^{R} \hat{h}(u) e^{\ii u y}\, \dd u\right) \phi_{m,k}(y)\, \dd y,
\end{align*} 
and by the dominated convergence theorem
\begin{equation*}
\int_{\R} \left(\int_{-R}^{R} \hat{h}(u) e^{\ii u y}\, \dd u\right) \phi_{m,k}(y)\, \dd y = \lim_{\epsilon \to 0} \int_{\R} \left(\int_{-R}^{R} \hat{h}(u) e^{\ii u y}\, \dd u\right)e^{- \pi \epsilon^{2}y^{2}} \phi_{m,k}(y) \dd y.
\end{equation*}
Since $ \hat{h} $ is continuous, $ \hat{h}(u)e^{-\pi \epsilon^{2}y^{2}}\phi_{m,k}(y) $ is integrable for $ (u,y) \in  [-R,R]\times \R $, and so we apply Fubini's theorem to obtain
\begin{align*}
    \int_{\R} \left(\int_{-R}^{R} \hat{h}(u) e^{\ii u y}\, \dd u\right) \phi_{m,k}(y)\, \dd y &= \lim_{\epsilon \to 0} \int_{-R}^{R} \left(\int_{\R} e^{- \pi \epsilon^{2}y^{2}} \phi_{m,k}(y) e^{\ii u y}\, \dd y\right) \hat{h}(u)\dd u\\
    & = \lim_{\epsilon \to 0} \int_{-R}^{R} \hat{h}(u) \overline{\hat{g}_{\epsilon} (u)} \, \dd u ,
\end{align*}
where $ \hat{g}_\epsilon $ is the Fourier transform of $y\mapsto e^{- \pi \epsilon^{2}y^{2}} \phi_{m,k}(y)$. As both $  e^{- \pi \epsilon^{2}y^{2}}$ and $\phi_{m,k}  $ belong to $ L^2(\R) $, the convolution theorem yields
\begin{align*}
   \hat{g}_\epsilon(u) &= (2\pi)^{-1} \left(\widehat{e^{-\pi \epsilon^{2}y^{2}}} * \hat{\phi}_{m,k}(y)\right)(u) =\\
   & = (2\pi)^{-1} 2^{-\frac{m}{2}} \left(\widehat{e^{-\pi \epsilon^{2}y^{2}}}  * \, e^{-\ii \frac{k}{2^{m}}y}\mathbf{1}_{[-2^{m}\pi,2^{m}\pi] }(y)
   \right)(u).
\end{align*}
Notice that 
$\widehat{e^{-\pi \epsilon^{2}y^{2}}}(u) = \epsilon^{-1} e^{-\frac{u^{2}}{4\pi\epsilon^{2}}} $ and that $  \epsilon^{-1}\int_{\R} e^{-\frac{u^{2}}{4\pi\epsilon^{2}}} \, \dd u = \int_{\R} e^{-\frac{y^{2}}{4\pi}} \dd y = 2\pi$. Namely, if we consider the mollifier $ \varphi(y)=(2\pi)^{-1} e^{-\frac{y^{2}}{4\pi}}  $ then
$$ \hat{g}_{\epsilon}(u) = 2^{-\frac{m}{2}} \epsilon^{-1} \left(\varphi(y/\epsilon)* \,e^{-\ii \frac{k}{2^{m}}y}\mathbf{1}_{[-2^{m}\pi,2^{m}\pi] }(y)\right)(u),$$
and for $ \epsilon\to 0 $ this sequence converges to $  2^{-\frac{m}{2}} e^{-\ii \frac{k}{2^{m}}u} \mathbf{1}_{[-2^{m}\pi,2^{m}\pi]}(u)= \hat{\phi}_{m,k}(u)$ for almost every $ u \in \R $. Finally, this shows that
\begin{align*}
     &\frac{1}{2\pi} \int_{\R} \left(\int_{-R}^{R} \hat{h}(u) e^{\ii u y}\, \dd u\right) \phi_{m,k}(y)\, \dd y = \lim_{\epsilon \to 0} \frac{1}{2\pi}\int_{-R}^{R} \hat{h}(u) \overline{\hat{g}_{\epsilon} (u)} \, \dd u  \\
     &=  \frac{1}{2\pi}\int_{-R}^{R} \hat{h}(u) \, \lim_{\epsilon \to 0}  \overline{\hat{g}_{\epsilon} (u)} \, \dd u =\frac{2^{-\frac{m}{2}}  }{2\pi} \int_{-R}^{R} \hat{h}(u) e^{\ii \frac{k}{2^{m}}u}\mathbf{1}_{[-2^{m}\pi,2^{m}\pi]}(u)\, \dd u,
\end{align*}
and since the support of the integrand is bounded, we can let $ R\to \infty $ and obtain the final statement.
\end{proof}

\section{Proofs of Proposition \ref{prop:GBM tail bound} and Corollary \ref{cor:gbm-decay-}} \label{sm:proofsLightTail}

\begin{proof}[Proof of Proposition \ref{prop:GBM tail bound}]
        We write
    $$I_r[\hat{h}](y) = \frac{1}{2\pi}\int_{-\infty}^\infty \hat{h}(u)e^{\ii y u} \dd u - \frac{1}{2\pi}\int_{|u|\ge r} \hat{h}(u)e^{\ii y u} \dd u, $$
    and study the modulus of each term separately. Since $\hat{h}$ is entire, by Cauchy estimates, assumption \eqref{eq:gaussian-growth-assumption} holds for all derivatives of $\hat{h}$. As a result, $\hat{h}$ verifies the assumptions of Lemma \ref{lem:fourier box value} and 
    $$\left|\frac{1}{2\pi}\int_{|u|\ge r} \hat{h}(u)e^{\ii y u} \dd u\right|\leq \frac{1}{\pi}\, \left| \Im\left[e^{\ii r  y}\hat{h}(r)\right]\right| |y|^{-1} + O(y^{-2}), \qquad y\neq 0. $$
    To estimate the integral in $[-\infty,\infty]$, consider $ \eta \in \R $ and the path
$[-\infty,\infty]\cup [\infty, \infty + \ii\eta]\cup [\infty+\ii\eta,-\infty+\ii\eta]\cup [-\infty+ \ii \eta,-\infty].$
As $ \hat{h} $ is analytic and vanishes at infinity, Cauchy's theorem implies that
\begin{align*}
\frac{1}{2\pi} \int_{-\infty}^{\infty} \hat{h}(u) e^{\ii u y} \dd u&=\frac{1}{2\pi} \int_{-\infty+\ii\eta}^{\infty+\ii\eta}\hat{h}(u) e^{\ii u y} \dd u. 
\end{align*}
Now taking the modulus and using the bound \eqref{eq:gaussian-growth-assumption} in the hypothesis we find that
\begin{align*}
\frac{1}{2\pi} \left|\int_{-\infty+\ii\eta}^{\infty+\ii\eta}\hat{h}(u) e^{\ii u y} \dd u \right|&= \frac{1}{2\pi} e^{-\eta y}\left|\int_{-\infty}^{\infty}\hat{h}(u+\ii\eta) e^{\ii u y} \dd u\right| \\
%&\leq \frac{1}{2\pi} e^{-\eta y} \int_{-\infty}^{\infty} |\hat{f}(u+\ii\eta)| \dd u \leq \frac{C}{2\pi} e^{-\eta y} \int_{-\infty}^{\infty} e^{-\gamma(u^2-\eta^2)} \dd u \\
&\leq \frac{C}{2\pi} e^{-\eta y+\gamma \eta^2}\int_{-\infty}^{\infty} e^{-\gamma u^2} \dd u= \frac{C}{2\sqrt{\pi\gamma}}e^{- \eta y +\gamma \eta^2}.
\end{align*}
Since $ \eta $ can be chosen arbitrarily, we can optimize this last bound. Given that $ \gamma>0 $, the function $ \eta\mapsto -\eta y + \gamma \eta^{2}$ has a global minimum at $ \eta=\frac{y}{2 \gamma} $ that is equal to $ -\frac{y^{2}}{4\gamma}$. We substitute and obtain
\begin{equation*}
  %\label{eq:est light tailed complex shift}
  \frac{1}{2\pi} \left|\int_{-\infty+\ii\eta}^{\infty+\ii\eta}\hat{h}(u) e^{\ii u y} \dd u \right|\leq \frac{C}{2\sqrt{\pi\gamma}}e^{-\frac{y^2}{4\gamma}}, \quad y \in  \R.
\end{equation*}
\end{proof}

\begin{proof}[Proof of Corollary \ref{cor:gbm-decay-}]
 From the expression of the FT of the GBM, we have that for all $\alpha \in \R$, $\hat{f}_\alpha(u\mid x)=\hat{f}(u+\ii \alpha \mid x)=e^{\alpha\mu+\gamma \alpha^2}e^{-\gamma u^2} e^{-\ii u(\mu+2\gamma \alpha)}$. Consequently, we can apply Proposition \ref{prop:GBM tail bound} with $C=e^{\alpha\mu+\gamma \alpha^2}$, $r=2^m\pi$, $y=y_{k,\mu}$ to find that
        $$|I_r[\hat{f}_\alpha]\left(y_k\right)|
				\le
				e^{\alpha\mu+ \gamma\alpha^2}
				\frac{e^{-y_{k,\mu}^2/4\gamma}}{2\sqrt{\pi \gamma}}
				+\frac{1}{\pi} \left|\Im\left[ \hat{f}(2^m \pi + \ii \alpha)\right]\right||y_{k,\mu}|^{-1} + O(y_{k,\mu}^{-2}). $$
Since $D^\alpha_{m,k} = 2^{-m/2} I_r[\hat{f}_\alpha](y_k)$, equation \eqref{eq:gbm-coeff-bound} follows immediately.
\end{proof}

\section{Particular case of Theorem \ref{teo:estimate fat tails}}\label{sup:residueTheorem}

Here we include a particular case of Theorem \ref{teo:estimate fat tails} where the considered singularities are just poles.

\begin{theorem}[Similar to {\cite[Theorem 4.3]{schiff1999laplaceSM}}]
\label{theorem:TF meromorphic}
Let $h\in L^1(\R)$, real-valued and piecewise $C^1$ with FT $ \hat{h} $ that is meromorphic in $ \mathbb{C} $ with a finite number of poles $ u_{1}, \ldots ,  u_{N}$ of orders $\rho_1,\ldots,\rho_N $.
%\in\mathbb{N}_{>0}$. 
Suppose also that
\begin{itemize}
  \item[i)]  $\hat{h}$ has at most polynomial growth, say $|\hat{h}(u)|\leq C (1+ |u|^q)$.
 \item[ii)] We can find $ \theta \in  (0,\frac{\pi}{2})$  such that 
 $$\sup_{u \in  \gamma(R, \theta)}\left(|\hat{h}( u)|+ |\hat{h}( -u)|\right), \quad   \sup_{u \in  \gamma(-R, \pi-\theta)}\left(|\hat{h}( u)|+ |\hat{h}( -u)|\right)$$
 tend to zero as $ R\to  + \infty $.
  \item[iii)] There are $ a,b >0 $ and $ \gamma^+, \gamma^- >0$ such that
\begin{equation*}
  %\label{eq:assumption singularities position}
  %\begin{cases}
  \left\{
  \begin{aligned}
    &\Im(u_n)>0 \implies \ \Im(u_n)\geq b, \text{ and } \Im(u_n)=b \implies  \Re(\rho_n)\leq  \gamma^+, \\
    &\Im(u_n)< 0 \implies \ \Im(u_n)\leq  -a,\text{ and }  \Im(u_n)=-a \implies \Re(\rho_n)\leq  \gamma^- .\\
  %\end{cases}
  \end{aligned}
  \right.
\end{equation*}
\end{itemize}
%Then there exist $C>0 $ and $ R>0 $ such that
%\begin{equation}
%    \label{eq:estimate fat tails}
%    \begin{aligned}
%      |f(x)|&\leq  C x^{\gamma^+ -1}e^{-bx} ,\qquad x>  R, \\ |f(x)|&\leq  C |x|^{\gamma^- -1}e^{-a|x|}, \qquad  x <  -R.
%    \end{aligned}
%  \end{equation}
Then there exists $C_+,C_- >0 $ such that
\begin{equation}
  \label{eq:est fourier inversion meromorphic}
    |I_{\infty}[\hat{h}](y)|\leq
    \begin{cases}
      C_+ y^{\gamma^+ - 1}e^{-by} ,\qquad y>  1,\\
      C_-|y|^{\gamma^- -1}e^{-a|y|} ,\qquad y<  -1.
    \end{cases}
\end{equation}
\end{theorem}

\begin{proof}\label{app:theorem:TF meromorphic}
As in Theorem \ref{teo:estimate fat tails}, $I_\infty[\hat{h}](y)$ is well-defined. From now on we assume that $y>0$, as the proof for $y<0$ is analogous. Therefore, for the case $y>0$, we work only with the poles whose imaginary part is positive. To estimate $I_\infty[\hat{h}](y)$ we consider the path $\gamma_R$, shown in \cref{fig:rectangleResidueTheorem}. Since $ \hat{h} $ is meromorphic in the domain enclosed by $\gamma_R$, Cauchy's residue theorem yields
\begin{equation*}
\begin{aligned}
I_\infty[\hat{h}](y) = -\frac{1}{2\pi}\lim_{R\to\infty}\int_{\gamma_R\setminus [-R,R]}  \hat{h}(u)e^{\ii yu}\dd u+\ii\left( \sum_{\Im(u_n)>0} \Res_{u=u_n}\left(  e^{\ii uy}\hat{h}(u)\right)\right).
\end{aligned}
\end{equation*}
As in the proof of Theorem \ref{teo:estimate fat tails} (Estimation along part $b)$), the integral along $\gamma_R\setminus [-R,R]$ tends to zero as $ R\to  \infty $. 

% To estimate $I_\infty[\hat{h}](y)$
% Under this set of assumptions, we can  define 
% $$ \frac{1}{2\pi}\int_{-\infty}^{\infty} \hat{h}(u)e^{\ii yu} \dd u :=  \frac{1}{2\pi}\int_{-r}^{r} \hat{h}(u)e^{\ii yu} \dd u + \frac{1}{2\pi}\int_{|u|\geq r} \hat{h}(u)e^{\ii yu} \dd u,  $$
% for some $ r>0 $. In fact, by Lemma \ref{lem:fourier box value} we have that the integral in $ \{ |u|\geq r \} $ is given by Equation \eqref{eq:fourier box value}, and since $ h \in L^{1}(\R) $, we have that $ \hat{h} $ is integrable in $ [-r,r] $.

% In order to prove the statement given by Equation \eqref{eq:estimation truncated meromorphic}, we assume that $ y>0 $ and consider the boundary $ \gamma_R $ of an isosceles trapezoid (see Figure \ref{fig:rectangleResidueTheorem}) given by
% \begin{equation}
%   \label{eq:proof TF meromorphi contour path}
%   \begin{aligned}
%   \gamma_R:= [-R,R]& \cup \{ R+ e^{\ii \theta } u :\, 0\leq u \leq \csc(\theta) R\}\\
%   &\cup \{ u + \ii R :\, -R(1+\cot(\theta))\leq u\leq R(1+\cot(\theta))\} \\
%   &  \cup \{ -R - e^{-\ii \theta } u :\, 0\leq u \leq \csc(\theta) R\}
% \end{aligned}
% \end{equation}
%   with $0<\theta < \frac{\pi}{2}$. We also assume that  $ R>0 $ is sufficiently big so that all poles are included in the region delimited by $ \gamma_R $.
\begin{figure}[!ht]
\centering
\begin{tikzpicture}[
  scale=0.4,
  wall/.style={black, thick},
  flow/.style={
    black, thick,
    postaction={decorate},
    decoration={
      markings,
      mark=at position 0.5 with {\arrow{Latex[length=3mm]}}
    }
  }
]

% parámetros
\def\L{6}
\def\h{6}
\def\a{1.5}

% vértices
\coordinate (A) at (0,0);
\coordinate (B) at (\L,0);
\coordinate (C) at (\L+\a,\h);
\coordinate (D) at (-\a,\h);

% --- contorno con flechas de orientación ---
\draw[flow] (A)--(B);
\draw[flow] (B)--(C);
\draw[flow] (C)--(D);
\draw[flow] (D)--(A);

% etiquetas
\node[below=4pt] at (A) {$-R$};
\node[below=4pt] at (B) {$R$};

\node[above=4pt] at (D) {$-R-R\cot(\theta)+R\ii$};
\node[above=4pt] at (C) {$R+R\cot(\theta)+R\ii$};

% puntos
\fill (0.8,1.0) circle(2pt) node[above right] {$ $};
\fill (1.2,4.0) circle(2pt) node[above right] {$ $};
\fill (4.0,2.0) circle(2pt) node[above right] {$ $};

\node at (2.5,3.2) {$\boldsymbol{\ddots}$};

\end{tikzpicture}
% \begin{tikzpicture}[
%   scale=0.4,
%   flow/.style={
%     black, thick,
%     postaction={decorate},
%     decoration={
%       markings,
%       mark=at position 0.5 with {\arrow{Latex[length=3mm]}}
%     }
%   }
% ]

% % parámetros
% \def\L{6}

% % vértices del cuadrado
% \coordinate (A) at (0,0);
% \coordinate (B) at (\L,0);
% \coordinate (C) at (\L,\L);
% \coordinate (D) at (0,\L);

% % contorno con flechas de orientación
% \draw[flow] (A)--(B);
% \draw[flow] (B)--(C);
% \draw[flow] (C)--(D);
% \draw[flow] (D)--(A);

% % etiquetas
% \node[below=4pt] at (A) {$-R$};
% \node[below=4pt] at (B) {$R$};

% \node[above=4pt] at (D) {$-R+R\ii$};
% \node[above=4pt] at (C) {$R+R\ii$};

% % puntos interiores
% \fill (2.0,1.0) circle(2pt);
% \fill (1.2,4.0) circle(2pt);
% \fill (5.0,5.0) circle(2pt);

% \node at (3.5,2.5) {$\boldsymbol{\iddots}$};

% \end{tikzpicture}
\caption{Integration path $\gamma_R$ in the upper half-plane. The black points are the poles contained in $\{\Im(u)>0\}$.}
\label{fig:rectangleResidueTheorem}
\end{figure}
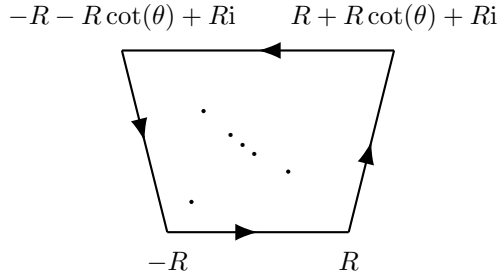

It remains to show that for $u_n$ with $\Im(u_n)>0$
$$  \left|\Res_{u=u_n}\left(  e^{\ii uy}\hat{h}(u)\right)\right|\leq  C |y|^{\gamma^+-1}e^{-by}, \qquad  y>1.$$
As $ \hat{h} $ is meromorphic, close to $ u_n$ with $ \Im(u_n)=b $, we have
$$  \hat{h}(u)=\sum_{k=0}^{\infty} a_k (u-u_n)^{k-\gamma^+}, \qquad  e^{\ii uy}= e^{\ii u_ny}\left( \sum_{k=0}^{\infty} \frac{(\ii y)^{k}}{k!}(u-u_n)^{k}\right).  $$
Consequently,
\begin{align*}
     \left|\Res_{u=u_n}(  e^{\ii uy}\hat{h}(u))\right| &= |e^{\ii u_ny}| \left| \sum_{k=0}^{\gamma^+-1} a_k \frac{(\ii y)^{\gamma^+-1-k}}{(\gamma^+-1-k)!} \right|\\
     & \leq  e^{-by}  \sum_{k=0}^{\gamma^+-1} |a_k| \frac{y^{\gamma^+-1-k}}{(\gamma^+-1-k)!}\leq  C y^{\gamma^+-1} e^{-by}.
\end{align*}
%Similarly, if the poles at $ \Im(u)=-a $ have order not bigger than $ \gamma^-\geq 1 $ then
%$$  \left|\Res_{u^-_n}(  e^{\ii uy}\hat{h}(u))\right|\leq  C |y|^{\gamma^{-}-1}e^{-a|y|}, \qquad  y<-1. $$

All in all, if the poles of $\hat{h}$ lie in $\Im(u)\geq b$ or $\Im(u)\leq -a$, then the estimation in \eqref{eq:est fourier inversion meromorphic} follows.
\end{proof}

% \section{Fourier coefficients decay for GBM dynamics}

% For GBM, $\hat{f}(w) = e^{-\ii w (A B + D) - C B w^2 }$ where $A = r-q-\frac12\sigma^2$, $B=T-t>0$, $C=\frac12\sigma^2>0$, $D = \log\left(\frac{S_0}{K}\right)$.
% The density coefficients are given by
% \begin{align*}
%     D_{m,k}^{\alpha} &= \frac{1}{2\pi} \int_{-2^m\pi}^{2^m\pi} \hat{f}(w+\alpha \ii) \frac{e^{\ii k \frac{w}{2^m}}}{2^{\frac{m}{2}}} \dd w \\
%     %&= \frac{1}{2\pi} 2^{-\frac{m}{2}}  \int_{-2^m\pi}^{2^m\pi} \exp\left(-\ii (w+\alpha \ii) (A B + D) - C B (w+\alpha \ii)^2 \right) e^{\ii k \frac{w}{2^m}} dw\\
%     %&= \frac{2^{-\frac{m}{2}}}{2\pi} e^{(AB+D)\alpha + CB \alpha^2} \int_{-2^m\pi}^{2^m\pi} \exp(-CB w^2) \exp(-\ii (AB+D+2CB\alpha) w) e^{\ii k \frac{w}{2^m}} dw \\
%     &= \frac{2^{-\frac{m}{2}}}{2\pi} e^{(AB+D)\alpha + CB \alpha^2} \int_{-2^m\pi}^{2^m\pi} e^{-CBw^2}e^{ \ii w \left( \frac{k}{2^m}-(AB+D+2CB\alpha)\right)} \dd w.
% \end{align*}
% At this stage, $ |D_{m,k}^{\alpha}| $ can be estimated by \eqref{eq:est coeff light tailed} applied with $\gamma=CB $, $ r=2^{m}\pi $ and
% $y =\beta:= \frac{k}{2^m}-(AB+D+2CB\alpha)$. Namely, we obtain that
% \begin{equation*}
%   \label{eq:coeff est density GBM}
%    |D_{m,k}^{\alpha}|\leq 2^{-\frac{m}{2}} e^{(AB+D)\alpha + CB \alpha^2}\left(e^{-\frac{\beta^2}{4CB}} \frac{1}{2\sqrt{\pi CB}} + \frac{2}{\pi}e^{-CB  4^m \pi^2}|\beta|^{-1}\right).
% \end{equation*}

\section{Generalized Hyperbolic and Normal Inverse Gaussian}\label{sup:GH_NIG}

In the following remarks, we will apply Corollary \ref{cor: estimate fat tails truncated} to the Generalized Hyperbolic (GH) and the Normal Inverse Gaussian (NIG) distributions.

\begin{remark}[GH]
The FT of the GH distribution (centered at $0$) is
$$ \hat{f}(u)= \frac{\gamma ^{\lambda }}{\left( a^{2}-(b -\ii u)^{2}\right)^{\frac{\lambda}{2}}}\frac {K_{\lambda }(\delta {\sqrt {a^{2}-(b -\ii u)^{2}}})}{K_{\lambda }(\delta \gamma )},$$
where $ K_\lambda $ is the modified Bessel function of the second kind of order $ \lambda $, and $\gamma = \sqrt{a^2-b^2}$. 
This function can be defined on $ \C \setminus\{\gamma_{u_1}, \gamma_{u_2}\} $ where $ u_1= \ii (a-b)$, $u_2 =-\ii (a+b)$, and the square root $\sqrt{(u-u_1)(u-u_2)}$ coincides with the usual root in case its argument is real.
Since $ a ^{2}-(b -\ii u)^{2}=(u-u_1)(u-u_2)  $ we have
$$ \hat{f}(u)= \gamma ^{\lambda} \left( (u-u_1)(u-u_2) \right)^{-\frac{\lambda}{2}}\frac {K_{\lambda }(\delta {\sqrt {(u-u_1)(u-u_2) }})}{K_{\lambda }(\delta \gamma )}.$$
To study the behaviour of the Bessel function at infinity, we use
$ K_\lambda (u)\sim \sqrt{\frac{\pi}{2u}}e^{-u}$, $|\arg(u)|<\frac{3}{2}\pi,$
which means that $ K_\lambda(u) $ decays exponentially in case $ \Re(u)>0 $, i.e. $ |\arg(u)|<\frac{\pi}{2}$. We can extend the usual real-valued $ \sqrt{u} $ in the complex plane so that $|\arg(\sqrt{u})| < \frac{\pi}{2} $ everywhere except for a branch cut where $ |\arg(u)|=\frac{\pi}{2}.$  Accordingly, $ |\hat{f}| $ decays exponentially along every direction, with the exception of the branch cuts where we have that
$$ |\hat{f}(u)|\leq \gamma ^{\lambda} \left| (u-u_1)(u-u_2) \right|^{-\frac{\lambda}{2}}\hspace{-5 pt}\sqrt{\frac{\pi}{2\sqrt{(u-u_1)(u-u_2)}}}\leq  C \left| (u-u_1)(u-u_2) \right|^{-\frac{\lambda}{2}-\frac{1}{4}}.$$
Next, by the definition of the modified Bessel function, when $\lambda \notin \Z$
$$ K_\lambda (u)= u^{-\lambda} F_{-\lambda}(u)- u^{\lambda}F_{\lambda}(u), $$
where both $ F_{-\lambda} $ and $ F_{\lambda} $ are entire functions of the form
$$ F_{\lambda}(u)= \sum_{m=0}^{\infty} a_m u^{2m}.$$
In particular, we see that $  F_{-\lambda}(\delta \sqrt{(u-u_1)(u-u_2)}) $ and $ F_{\lambda}(\delta\sqrt{(u-u_1)(u-u_2)}) $ are entire functions. In light of this, we have
$$ \hat{f}(u)= \left((u-u_1)(u-u_2)\right)^{-\lambda} F_{-\lambda}(\delta\sqrt{(u-u_1)(u-u_2)}) -F_{\lambda}(\delta\sqrt{(u-u_1)(u-u_2)}),$$
i.e., $ \hat{f} $ can be written as the sum of an entire function and a function with algebraic singularities of order $\lambda$. Consequently, we can apply Corollary \ref{cor: estimate fat tails truncated}, and by \eqref{eq:density_bound} we find $ C_+, C_->0 $ such that (for $y_k:=k/2^m$ and $r=2^m\pi$) %{\color{red}{JG: say sth like for $x=k/2^m$ and $r=2^m\pi$, $\Im\left[e^{-\ii xr} \hat{f}(-r)\right]=\Im[\hat{f}(-r)]$}, and consider also the missing absolute value of the imaginary part. Apply this corrections also below}
 \begin{equation*}
    |I_r[\hat{f}](y_k)|\leq
    \begin{cases}
      C_+ y_k^{\lambda-1}e^{-(a-b)y_k}  + \frac{1}{\pi y_k}\left|\Im[\hat{f}(r)]\right| + O(y_k^{-2}) ,\quad y_k>  1,\\
      C_- |y_k|^{\lambda-1}e^{-(b+a)|y_k|}  +  \frac{1}{\pi |y_k|}\left|\Im[\hat{f}(r)]\right| + O(y_k^{-2}),\, y_k<  -1.
    \end{cases}
\end{equation*}
This estimation also holds when $\lambda \in \mathbb{Z}$. In case $\lambda =0$, $K_\lambda(u)$ has a logarithmic branch point of order $0$ at the origin, and Corollary \ref{cor: estimate fat tails truncated} applies with maximum order equal to zero. In case $\lambda\in \Z \setminus \{0\}$, then $K_\lambda(u) =u^{-|\lambda|}g_0(u) + u^{|\lambda|} \log(u)  g_1(u) + H(u) $, where $g_0,g_1$ and $H$ are entire functions. Accordingly, when $\lambda >0$, $\hat{f}$ has a logarithmic contribution of order $0$, and an algebraic one of order $\lambda$, which dominates the one of order $0$. In case $\lambda <0$, $\hat{f}$ only has a logarithmic contribution of order $\lambda$. In any case, we can apply  Corollary \ref{cor: estimate fat tails truncated} with maximum order equal to $\lambda$ and get the estimation above.
\end{remark}

\begin{remark}[NIG]
 This is a special case of the GH distribution with $\lambda=-\frac12$. %Consequently, we can find $C>0$ such that
 %\begin{equation*}
 %   \left|\frac{1}{2\pi}\int_{-r}^{r} \hat{f}(u)e^{\ii ux}du\right|\leq
 %   \begin{cases}
 %     C \left(x^{-\frac32}e^{-(a-b)x}  + x^{-1} e^{-r \delta \cos(\pi/8) }\right),\qquad x>  1,\\
  %    C \left(|x|^{-\frac32}e^{-(b+a)|x|}  + |x|^{-1} e^{-r \delta \cos(\pi/8) }\right),\qquad x<  -1.
  %  \end{cases}
%\end{equation*}
\end{remark}

\section{Bounds: numerical example}\label{sup:boundsNumExam}

In this section, we aim to illustrate with a numerical example the estimates provided by section \ref{sec:Asymptotic decay of the payoff and density coefficients with respect to the translation parameters} regarding the decay of the density and payoff coefficients. The example concerns the valuation of a cash-or-nothing (CON) call option under the NIG model. The NIG (see \cite{bayer2024quasiSM}) and option parameters are $\alpha_{\rm{NIG}}=5$, $\beta_{\rm{NIG}}=3$, $\Delta_{\rm{NIG}}=1$, $S_0=K=100$, $r_0=0.1$, $T=1$. The parameters of the damped SWIFT method are $m=2$ and $\alpha = 1.29$. For these values of $m$ and $\alpha$, $\varepsilon_V = 0.0251$ and $\varepsilon_D = 6.72 \cdot 10^{-6}$.

We begin with the study of the decay of the Fourier coefficients of the density, illustrated in Figure \ref{fig:density_bounds_decay}. In the right tail, since $\alpha$ is positive (we are pricing a call CON), the density has been undamped, so the decay is initially slower than in the left tail, where damping is applied to the density. As predicted in section \ref{sec:Asymptotic decay of the payoff and density coefficients with respect to the translation parameters}, two decay regimes are observed in each tail: first an exponential regime and eventually a linear one. The transition between the two regimes is governed by the value of $\varepsilon_D$ relative to the exponentially decaying term. In the plot, the coefficients with exponential decay are marked with stars, whereas those exhibiting linear decay are drawn with a continuous line. 
%In the right tail, the region with exponential decay approximately covers the coefficients $k=0,\ldots,64$. In this region, the term with linear decay in \eqref{eq:est truncated fourier inversion} is negligible compared with the exponentially decaying term. Beyond this point, the situation reverses: the term in \eqref{eq:est truncated fourier inversion} with exponential decay becomes negligible compared with the linearly decaying term, and therefore the decay of the coefficients becomes linear. 
%In both cases, 
It can be observed that the bounds provided in section \ref{sec:Asymptotic decay of the payoff and density coefficients with respect to the translation parameters} closely match the actual decay of the coefficient modulus, except very close to the mean, as expected due to the asymptotic nature of the theoretical results. In the left tail, the situation is similar. The only difference is that, due to the damping of the density, the region with linear decay begins earlier because the exponentially decaying terms are smaller than in the right tail.

Finally, in Figure \ref{fig:payoff_bounds_decay} we present the actual decay of the payoff coefficients together with the bounds. Since the payoff of a call CON is zero to the left of zero, in the left tail, the decay of the payoff coefficients is purely linear. In the right tail, there is an initial region of exponential decay.
%, approximately for $k=0,\ldots,20$. 
This region ends earlier than in the case of the density because $\varepsilon_V \gg \varepsilon_D$. Beyond this point, the decay becomes linear again. In both decay regimes, the bound is very tight.

\begin{figure}[!htbp]
 \centering
 \includegraphics[scale=0.3]{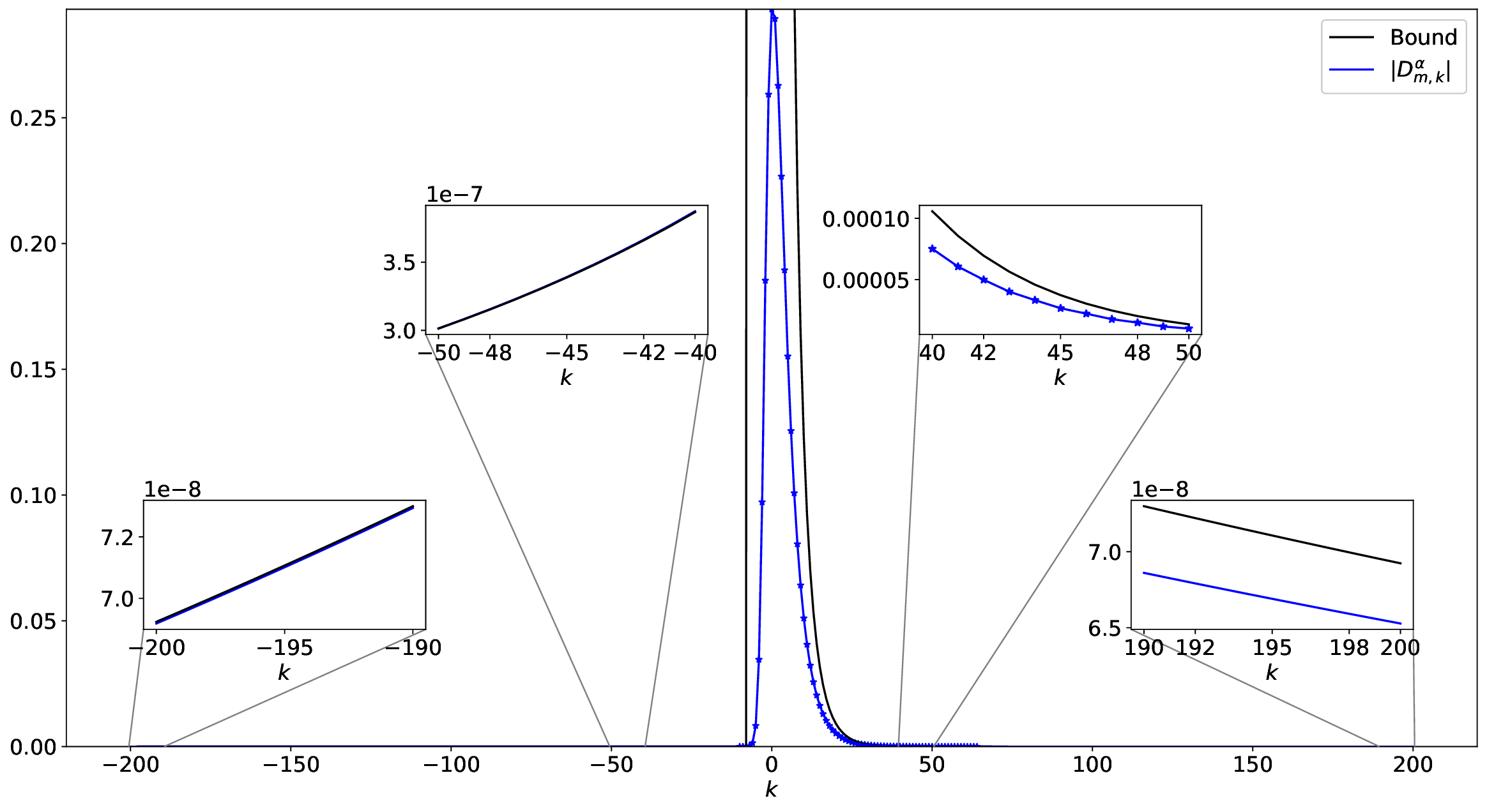}
 \caption{Decay of the absolute value of the Fourier coefficients of the density.}
 \label{fig:density_bounds_decay}
\end{figure}

\begin{figure}[!htbp]
 \centering
 \includegraphics[scale=0.3]{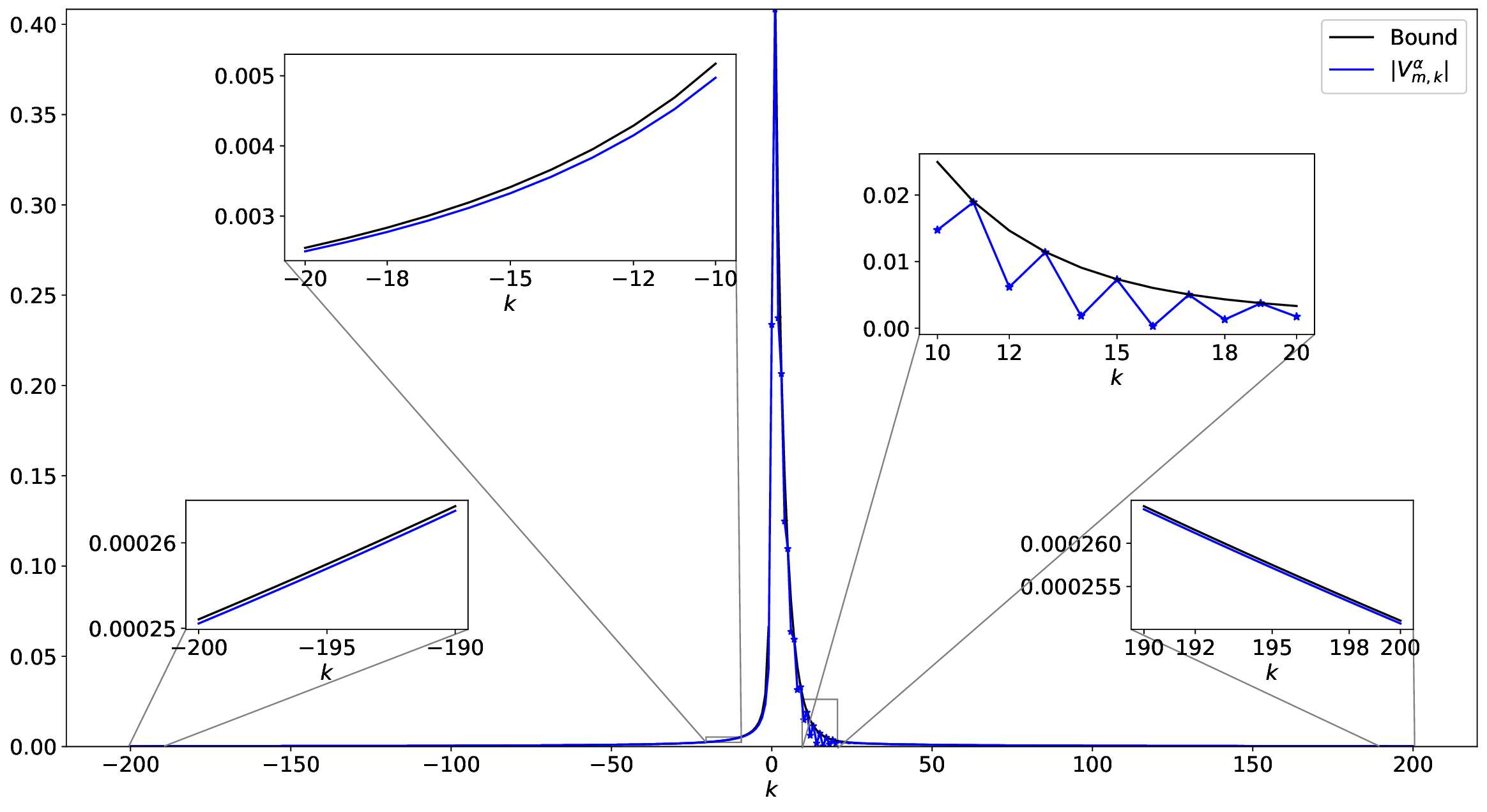}
 \caption{Decay of the absolute value of the Fourier coefficients of the payoff.}
 \label{fig:payoff_bounds_decay}
\end{figure}

\end{document}